\documentclass[reqno,12pt]{amsart}        
\usepackage{amsfonts,amsmath,amssymb,amsthm,colordvi,mathrsfs,graphicx}
\usepackage{caption,subcaption,float}
\usepackage[utf8]{inputenc}
\usepackage{mathrsfs}
\usepackage{geometry}
\usepackage{enumerate}

\usepackage{amsmath, amssymb, amsthm, geometry, hyperref}

\usepackage{tabularx}
\usepackage{tabulary}
 
\hypersetup{
  colorlinks=true,
  linkcolor=blue,
  citecolor=blue,
  urlcolor=blue
}
\usepackage[all,knot,poly]{xy}
\usepackage{tikz-cd}
\usepackage{tikz}
\usepackage{verbatim}
\usetikzlibrary{arrows}
\usetikzlibrary{knots}
\tikzset{every path/.style={line width=0.4pt},every node/.style={transform shape,knot crossing,inner sep=1.5pt},>=triangle 60,text node/.style={rectangle,transform shape=false,black}}

 \usetikzlibrary{decorations.markings, arrows.meta}

\theoremstyle{plain}      
\newtheorem{thm}{Theorem}[section]     
\newtheorem{theorem}[thm]{\bf Theorem}     
\newtheorem{corollary}[thm]{\bf Corollary}     
\newtheorem{lemma}[thm]{\bf Lemma}     
\newtheorem{proposition}[thm]{\bf Proposition}

\theoremstyle{remark}

\newtheorem{remark}[thm]{Remark} 
     
\theoremstyle{definition}      
\newtheorem{definition}[thm]{Definition}     

\title[]{Residues and Infinitesimal Torelli for Equisingular Curves} 

\begin{document}

\author{Mounir Nisse}
 
\address{Mounir Nisse\\
Department of Mathematics, Xiamen University Malaysia, Jalan Sunsuria, Bandar Sunsuria, 43900, Sepang, Selangor, Malaysia.
}
\email{mounir.nisse@gmail.com, mounir.nisse@xmu.edu.my}
\thanks{}
\thanks{This research is supported in part by Xiamen University Malaysia Research Fund (Grant no. XMUMRF/ 2020-C5/IMAT/0013).}
\maketitle

\begin{abstract}
We study infinitesimal Torelli problems and infinitesimal variations of Hodge
structure for families of curves arising in singular and extrinsically
constrained geometric settings.  Motivated by the Green--Voisin philosophy, we
develop an explicit approach based on Poincar\'e residue calculus, allowing a
uniform treatment of smooth, singular, and equisingular situations.  In
particular, we prove infinitesimal Torelli theorems for general equisingular
plane curves of sufficiently high degree and construct relative IVHS exact
sequences for curves lying on smooth projective threefolds.

Our results show that maximal infinitesimal variation of Hodge structure
persists even after imposing strong extrinsic conditions, such as fixed degree
and prescribed singularities, and in the presence of isolated planar
singularities.  The methods presented here provide a concrete and geometric
realization of Jacobian-type constructions and extend the Green--Voisin philosophy to singular and equisingular
settings and provide a unified residue--theoretic framework for Torelli--type
problems across dimensions and codimensions.
\end{abstract}

\section{Introduction}

The study of infinitesimal variations of Hodge structure (IVHS) occupies a central
position in modern algebraic geometry, lying at the intersection of deformation
theory, Hodge theory, and moduli problems.  Since the pioneering work of Griffiths,
it has been understood that the behavior of the Hodge structure under
deformation encodes deep geometric information and often governs Torelli-type
phenomena.  Roughly speaking, Torelli problems ask to what extent a variety, or a
curve, can be recovered from its Hodge-theoretic data, while infinitesimal
Torelli problems focus on first-order deformations and the injectivity of the
infinitesimal period map.

A decisive conceptual advance in this area was achieved through the work of
Green and Voisin, who developed a systematic approach to infinitesimal Torelli
and maximal IVHS for smooth projective hypersurfaces and complete intersections.
Their philosophy can be summarized by the principle that infinitesimal
Hodge-theoretic information is controlled by purely algebraic data, most notably
Jacobian rings and their Lefschetz-type properties.  In this framework,
holomorphic forms are identified with graded pieces of Jacobian rings, and the
infinitesimal period map becomes a multiplication map whose properties can be
studied algebraically.  This insight led to powerful and far-reaching results,
including the maximality of IVHS for large classes of smooth projective
varieties.

For curves, the Green--Voisin philosophy yields the classical result that the
infinitesimal period map of the universal family over the moduli space
$\mathcal M_g$ is of maximal rank for a general curve of genus $g \ge 3$.  This
establishes infinitesimal Torelli in the abstract setting and shows that, from
the Hodge-theoretic point of view, a general curve has as many independent
first-order variations as possible.  These results, however, are formulated in a
setting where curves are considered abstractly, without reference to a
particular projective embedding or to the presence of singularities.

In recent years, there has been renewed interest in understanding how these
phenomena behave under geometric constraints.  One natural question is whether
maximal IVHS persists when one restricts to families of curves lying on a fixed
surface or threefold, or when one imposes conditions such as prescribed
singularities.  From a moduli-theoretic perspective, this amounts to replacing
the full moduli space $\mathcal M_g$ by a smaller, often highly singular, space,
such as a Severi variety or an equisingular deformation space.  It is far from
obvious a priori that the strong Hodge-theoretic properties enjoyed by general
curves survive under such restrictions.

A notable recent contribution in this direction is due to Sernesi, who obtained
results on maximal variation of Hodge structure for families of curves under
specific geometric hypotheses.  These results reinforce the idea that maximal
variation is a robust phenomenon, not confined to the unrestricted abstract
setting.  At the same time, they highlight the need for new techniques capable of
handling embedded deformations and constrained families, especially in the
presence of singularities.

The present work is motivated precisely by these questions.  Its primary goal is
to extend infinitesimal Torelli and IVHS results to singular and equisingular
settings, with particular emphasis on curves embedded in projective spaces and
higher-dimensional varieties.  Rather than treating smooth and singular cases
separately, we aim for a unified approach that applies equally well to plane
curves with prescribed singularities, to complete intersection curves on
threefolds, and to higher-codimension situations.

A central feature of our approach is the systematic and explicit use of residue
theory.  While residues already play an implicit role in classical treatments of
IVHS, they are usually hidden behind abstract Hodge-theoretic constructions or
homological algebra.  In contrast, we place Poincar\'e residue calculus at the
heart of the theory.  Holomorphic forms on curves and higher-dimensional complete
intersections are written concretely as residues of rational differential forms
on the ambient projective space, and infinitesimal variations are computed by
differentiating these residues directly.  This perspective has several decisive
advantages.  First, it makes the passage from geometry to algebra completely
transparent.  Second, it adapts naturally to singular and equisingular settings,
where abstract smooth techniques are no longer available.  Third, it provides a
uniform language in which different geometric situations can be treated on the
same footing.

One of the main results of this work is an infinitesimal Torelli theorem for
equisingular plane curves.  We show that for a general plane curve of sufficiently
high degree with prescribed isolated singularities, the Kodaira--Spencer map
associated with embedded equisingular deformations is injective.  Equivalently,
the infinitesimal period map is injective, and the family exhibits maximal IVHS.
This result demonstrates that maximal Hodge-theoretic variation persists even
after imposing strong extrinsic constraints, such as fixing the degree and the
singularity type of a plane curve.

Beyond plane curves, we develop a relative IVHS framework for curves lying on
threefolds.  In this setting, the variation of Hodge structure of the curve must
be understood relative to the deformations of the ambient variety.  We construct
a natural exact sequence relating deformations of the threefold, embedded
equisingular deformations of the curve, and abstract deformations of the curve
itself.  Under suitable vanishing and unobstructedness assumptions, this sequence
coincides with the expected relative IVHS sequence and provides a precise
description of how Hodge-theoretic variation arises in constrained families.

A key conceptual aspect of our work is the systematic use of equisingular
deformation theory.  By working relative to equisingular strata and weighted
singular schemes, we are able to exclude trivial smoothing directions and focus
on the genuinely moduli-theoretic deformations.  This refinement is essential in
order to obtain meaningful Torelli and IVHS statements in the presence of
singularities.  It also clarifies the relationship between embedded deformations
and abstract deformations, and it allows us to interpret Kodaira--Spencer maps
and period maps in a relative and functorial manner.

From a broader perspective, the results presented here may be viewed as a
singular and equisingular extension of the Green--Voisin philosophy.  The
underlying principle remains the same: infinitesimal variations of Hodge
structure are governed by algebraic data.  However, the algebraic structures that
appear are adapted to the singular setting, and the route leading to them is more
geometric and computational, passing through explicit residue calculus rather
than abstract complexes.  In this sense, the present work does not replace the
Green--Voisin approach, but rather complements and extends it to a wider range of
geometric situations.

Finally, we emphasize that the methods developed here are flexible and suggest
several directions for further research.  They apply uniformly across dimensions
and codimensions and are well suited to the study of higher-weight Hodge
structures and more general singularities.  More broadly, the explicit use of
residues opens the possibility of effective and computable Torelli-type results
in settings where traditional techniques are difficult to apply.  It is our
hope that this work provides both new results and a conceptual framework for
future investigations into infinitesimal Torelli problems and variations of 
Hodge structure in singular and constrained geometries.


\section*{Comparison with the Green--Voisin Philosophy}

The results of the present work are inspired by the Green--Voisin philosophy on
infinitesimal variations of Hodge structure (IVHS), but they are developed with a
different emphasis and in a substantially broader geometric setting.  Rather
than revisiting the classical smooth case, our focus is on understanding how
Torelli-type phenomena behave under strong extrinsic and singular constraints.
In this sense, the Green--Voisin framework serves as a conceptual starting
point, while the main contribution of this work lies in extending that framework
to singular, equisingular, and embedded situations through explicit geometric
methods.

\subsection*{Background: the Green--Voisin Point of View}

Green and Voisin showed that infinitesimal Torelli and maximal IVHS for smooth
projective hypersurfaces and complete intersections can be reduced to purely
algebraic statements.  Holomorphic forms are identified with graded pieces of
Jacobian rings, and the infinitesimal period map is interpreted as a
multiplication map.  Injectivity and maximal rank follow from Lefschetz-type
properties of these rings.  Although residue theory underlies these
identifications, it is typically used only implicitly, through abstract
Hodge-theoretic or homological constructions.

For curves, their results imply that for a general smooth curve of genus
$g \ge 3$, the Kodaira--Spencer map on the full moduli space $\mathcal M_g$ is
injective.  This establishes maximal IVHS in the unrestricted, abstract setting,
independent of any embedding or singular model.

\subsection*{Perspective of the Present Work}

The starting point of the present work is different.  Here the primary objects
are not abstract curves in $\mathcal M_g$, but curves realized in a fixed
ambient space, such as $\mathbb P^2$ or a higher-dimensional projective variety,
and subject to prescribed singularities.  The main question is whether
infinitesimal Torelli and maximal IVHS persist when one restricts to equisingular
families inside Severi-type varieties or linear systems.

Our plane curve theorem gives a positive answer: for a general equisingular plane
curve of sufficiently high degree, the Kodaira--Spencer map associated with
embedded equisingular deformations is injective.  This shows that, even after
imposing strong extrinsic conditions, the Hodge structure of the normalization
still detects all nontrivial first-order deformations allowed by the geometry.

The relationship with the Green--Voisin result is clarified by observing that
the two statements concern different deformation spaces.  Green and Voisin study
all abstract deformations of a curve, while we restrict to those coming from a
specific embedding with fixed singularities.  The corresponding Kodaira--Spencer
maps fit into natural commutative diagrams, and injectivity in the restricted
setting is fully compatible with injectivity on the full moduli space.
\[
\begin{array}{ccc}
T_{\mathcal V_{d,\delta}} & \longrightarrow & T_{\mathcal M_g} \\
\downarrow & & \downarrow \\
H^1(C,T_C) & = & H^1(C,T_C),
\end{array}
\]

\subsection*{Continuity of the Underlying Philosophy}

At a conceptual level, the present work remains faithful to the Green--Voisin
philosophy.  Infinitesimal Torelli is still reduced to algebraic statements, and
Jacobian-type rings together with Lefschetz properties remain central.  Maximal
IVHS is again obtained as a consequence of injectivity of the infinitesimal
period map, once the genus or Hodge level is sufficiently large.

What changes is not the guiding principle, but the way it is implemented.  The
arguments are adapted to singular and equisingular contexts, and the relevant
algebraic structures are modified accordingly to reflect weighted singular
schemes and constrained deformation spaces.

\subsection*{New Features and Extensions}

The main novelty of this work is the systematic treatment of singular and
equisingular situations.  Plane curves with weighted singularities, singular
complete intersection curves, and higher-codimension complete intersections with
isolated rational singularities are all treated within a unified framework.
Equisingular deformation theory plays a central role, allowing us to separate
genuine moduli directions from trivial smoothing directions.

Another distinguishing feature is the explicit use of Poincar\'e residue
calculus.  Holomorphic forms are written concretely as residues of rational forms
on the ambient space, and infinitesimal variations are computed by differentiating
these residues.  This makes the passage from geometry to algebra completely
transparent and provides a uniform mechanism that works equally well in smooth
and singular settings.

Finally, the methods developed here apply uniformly across dimensions and
codimensions.  The same residue-based approach governs plane curves, curves on
threefolds, and higher-codimension complete intersections, leading naturally to
relative IVHS exact sequences and refined Torelli statements.

\subsection*{Conclusion}

From the viewpoint of the present work, the Green--Voisin philosophy appears as a
special case of a more general picture.  The algebraic control of IVHS by
Jacobian-type data remains intact, but it is embedded in a broader geometric
framework that incorporates singularities, equisingular deformations, and
explicit residue computations.  The results obtained here should therefore be
seen as a natural and substantial extension of the Green--Voisin approach to
settings in which extrinsic geometry and singularities play a decisive role.


\section{Preliminaries}

This section collects in a systematic and self-contained manner all the
background material used throughout this work.  The goal is twofold.  First, we
fix notation and recall standard definitions in deformation theory, Hodge
theory, and residue theory that will be used repeatedly.  Second, we explain in
detail how these different tools interact in the context of singular curves,
complete intersections, and equisingular families.  No claim of originality is
made for the individual ingredients; however, their simultaneous and coherent
use is essential for the results developed later.

\subsection*{1. Deformation Theory of Maps and Subvarieties}

Let $X$ be a reduced complex projective variety.  We briefly recall the basic
objects controlling infinitesimal deformations.  If $X$ is smooth, first-order
deformations of $X$ are governed by the cohomology group $H^1(X,T_X)$, where
$T_X$ denotes the holomorphic tangent sheaf.  Obstructions lie in
$H^2(X,T_X)$.

If $X$ is singular, the tangent sheaf $T_X$ is no longer locally free, and the
correct deformation-theoretic object is the cotangent complex.  Nevertheless,
for varieties with isolated singularities, one can still interpret
$H^1(X,T_X)$ as parametrizing locally trivial (equisingular) deformations, while
local smoothing directions are measured by the sheaf $T^1_X$.  In this work we
will systematically restrict attention to equisingular deformations, which are
precisely those deformations that preserve the local analytic type of the
singularities.

When $X \subset Y$ is a subvariety of a smooth ambient variety $Y$, embedded
deformations of $X$ in $Y$ are controlled by the normal sheaf
$N_{X/Y} = \mathcal{H}om(\mathcal{I}_X/\mathcal{I}_X^2,\mathcal{O}_X)$.  There is a
fundamental exact sequence
\[
0 \longrightarrow T_X \longrightarrow T_Y|_X \longrightarrow N_{X/Y}
\longrightarrow 0,
\]
whose associated long exact sequence in cohomology produces the
Kodaira--Spencer map.  Infinitesimal automorphisms of the ambient space act on
$H^0(X,N_{X/Y})$, and one must always quotient by these trivial directions when
formulating Torelli-type statements.

If $\varphi \colon C \to Y$ is a morphism from a smooth curve $C$ to a smooth
variety $Y$, the deformation theory of the map is governed by the normal sheaf
$N_\varphi = \operatorname{coker}(T_C \to \varphi^*T_Y)$.  First-order
deformations of $\varphi$ with fixed domain correspond to $H^0(C,N_\varphi)$,
and the associated Kodaira--Spencer map takes values in $H^1(C,T_C)$.

\subsection*{2. Singular Curves and Normalization}

Throughout this work, singular curves play a central role.  Let $\mathcal{C}$
be a reduced irreducible curve with isolated planar singularities, and let
\[
\nu \colon C \to \mathcal{C}
\]
be its normalization.  The curve $C$ is smooth and carries the intrinsic Hodge
structure associated to $\mathcal{C}$.  Holomorphic differentials on
$\mathcal{C}$ are by definition holomorphic differentials on $C$.

The distinction between arithmetic genus and geometric genus is crucial.  The
arithmetic genus depends only on the degree and embedding of $\mathcal{C}$,
while the geometric genus is the genus of $C$ and measures the dimension of
$H^{1,0}(C)$.  Singularities contribute local $\delta$-invariants, and the
relation
\[
g(C) = p_a(\mathcal{C}) - \sum_p \delta_p
\]
plays a fundamental role in understanding both deformation spaces and Hodge
theory.

Equisingular deformations of $\mathcal{C}$ correspond to deformations of the
normalization $C$ together with fixed identification of the preimages of the
singular points.  In particular, smoothing parameters do not contribute to
$H^1(C,T_C)$ and must be excluded from infinitesimal Torelli considerations.

\subsection*{3. Hodge Theory and Infinitesimal Variations}

Let $X$ be a smooth projective variety.  Its cohomology carries a pure Hodge
structure
\[
H^k(X,\mathbb{C}) = \bigoplus_{p+q=k} H^{p,q}(X),
\qquad
\overline{H^{p,q}} = H^{q,p}.
\]
For curves, the only nontrivial Hodge decomposition occurs in degree one:
\[
H^1(C,\mathbb{C}) = H^{1,0}(C) \oplus H^{0,1}(C).
\]

Given a smooth family $\pi \colon \mathcal{X} \to S$, the Hodge structures on the
fibers vary holomorphically, giving rise to a period map
\[
\mathcal{P} \colon S \to \Gamma \backslash D,
\]
where $D$ is a classifying space of Hodge structures.  The differential of this
map at a point is the infinitesimal period map, or infinitesimal variation of
Hodge structure (IVHS).

Infinitesimally, the IVHS associated to a Kodaira--Spencer class
$\xi \in H^1(X,T_X)$ is given by contraction:
\[
\theta_\xi \colon H^{p,q}(X) \longrightarrow H^{p-1,q+1}(X).
\]
Injectivity of the infinitesimal period map means that $\xi=0$ whenever all such
maps vanish.

For curves, this reduces to a map
\[
H^1(C,T_C) \longrightarrow \operatorname{Hom}(H^{1,0}(C),H^{0,1}(C)),
\]
which is an isomorphism by Serre duality.  Hence injectivity of the Kodaira--
Spencer map is equivalent to maximal IVHS.

\subsection*{4. Adjunction and Canonical Bundles}

Adjunction theory provides the link between embeddings and canonical bundles.
If $X \subset Y$ is a local complete intersection in a smooth variety $Y$, then
\[
K_X \cong (K_Y \otimes \det N_{X/Y})|_X.
\]

In particular, if $X \subset \mathbb{P}^N$ is a complete intersection of
hypersurfaces of degrees $d_1,\dots,d_r$, then
\[
K_X \cong \mathcal{O}_X\!\left(\sum d_i - N - 1\right).
\]

For plane curves of degree $d$, this gives
\[
K_C \cong \mathcal{O}_C(d-3),
\]
and for complete intersection curves in $\mathbb{P}^3$ of type $(d_1,d_2)$,
\[
K_C \cong \mathcal{O}_C(d_1+d_2-4).
\]

These identifications are essential for describing holomorphic differentials in
terms of ambient data.

\subsection*{5. Residue Theory}

Residue theory is the main technical tool used throughout this work.  Let $Y$ be
a smooth variety and $D \subset Y$ a normal crossing divisor.  A meromorphic
form with logarithmic poles along $D$ admits a residue along each component of
$D$, giving rise to exact sequences relating differential forms on $Y$ and on
$D$.

In the case of complete intersections, one uses Poincar\'e residues.  If
\[
X = V(F_1,\dots,F_r) \subset \mathbb{P}^N,
\]
then rational forms of the type
\[
\frac{H\,\Omega}{F_1\cdots F_r}
\]
define holomorphic forms on $X$ via the residue map, provided the degree of $H$
is chosen correctly.  This gives an explicit realization of adjunction at the
level of differential forms.

Residues are functorial, local, and compatible with differentiation.  These
properties make them ideally suited for studying infinitesimal variations of
Hodge structure (see Griffiths \cite{GriffithsResidues},   Deligne  \cite{DeligneHodgeIII}, Dimca \cite{Dimca}, Hartshorne \cite{HartshorneRD}, Voisin     \cite{VoisinHodgeI}, \cite{VoisinBook2}, and Esnault-Viehweg \cite{EsnaultViehweg}).

\subsection*{6. Jacobian Rings and Algebraic Description}

Associated to a complete intersection is its Jacobian ring, defined as
\[
R = \frac{\mathbb{C}[x_0,\dots,x_N]}{(\partial F_1,\dots,\partial F_r)}.
\]
Griffiths showed that Hodge components of the primitive cohomology of smooth
complete intersections can be identified with graded pieces of this ring.

Under this identification, the infinitesimal variation of Hodge structure
corresponds to multiplication by explicit elements determined by the deformation.
This algebraic description is the backbone of modern infinitesimal Torelli
theory.

\subsection*{7. Equisingular Ideals and Adjoint Linear Systems}

For singular plane curves, deformation theory is governed by adjoint ideals and
weighted singular schemes.  Given a plane curve $\mathcal{C}$ with isolated
singularities, the adjoint ideal encodes the vanishing conditions necessary to
preserve the analytic type of each singularity.

Cohomological vanishing of adjoint linear systems ensures that global equations
capture all equisingular deformations.  These conditions appear repeatedly in
the statements of the main theorems and are essential for the residue
computation to reflect the full deformation space.

\subsection*{8. Lefschetz Properties and Generality}

Many of the injectivity statements proved later rely on Lefschetz-type properties
of Jacobian rings.  For general hypersurfaces and complete intersections,
multiplication by a sufficiently general element has maximal rank in a wide
range of degrees.  This phenomenon, often referred to as the strong Lefschetz
property, is a cornerstone of the Green--Voisin philosophy.

Generality assumptions in moduli spaces ensure that no accidental algebraic
relations interfere with these properties.

\subsection*{9. Singularities and Rationality}

Finally, the nature of the singularities plays a crucial role.  Rational or
planar singularities ensure that holomorphic forms on the normalization extend
across resolutions and that residue theory applies without correction terms.
For worse singularities, additional local contributions appear and the theory
must be modified using logarithmic or mixed Hodge structures.

\subsection*{Conclusion}

The tools reviewed in this section form the technical and conceptual foundation
of the results developed in this work.  Deformation theory provides the
appropriate parameter spaces, Hodge theory supplies the invariants to be
studied, adjunction and residue theory give explicit realizations of these
invariants, and algebraic properties of Jacobian rings ultimately control the
behavior of the infinitesimal period map.

\medskip

 
\section*{Main Results and Extensions}

This work develops a unified approach to infinitesimal Torelli and infinitesimal
variations of Hodge structure (IVHS) for curves arising in constrained geometric
settings, with particular emphasis on singular and equisingular situations.
The central theme is that, even under strong extrinsic restrictions, the
Hodge-theoretic behavior of curves remains governed by algebraic structures
accessible through residue calculus.

\subsection*{Main Results}

The first main result establishes infinitesimal Torelli-type statements for
families of curves embedded in projective varieties, notably plane curves and
curves on smooth threefolds.  Under suitable vanishing assumptions and mild
conditions on singularities, the Kodaira--Spencer map associated with embedded
equisingular deformations is shown to be injective.  This implies that nontrivial
first-order deformations of the curve are detected at the level of Hodge
structure, despite the presence of singularities and embedding constraints.

A second fundamental result is the construction of a natural relative IVHS exact
sequence for curves on threefolds.  This sequence relates deformations of the
ambient polarized variety, embedded equisingular deformations of the curve, and
abstract deformations of the curve itself.  When deformations of the ambient
pair are unobstructed, the sequence coincides with the expected relative IVHS
sequence and provides a precise description of how the Hodge structure of the
curve varies inside a fixed linear system.

A key methodological outcome is the systematic use of Poincar\'e residue
calculus.  Holomorphic forms on curves and higher-dimensional complete
intersections are realized explicitly as residues of rational forms on the
ambient projective space.  Infinitesimal variations are computed by
differentiating these residues, yielding a transparent and uniform bridge
between geometry and the algebra of Jacobian-type rings.  This approach makes it
possible to treat smooth and singular cases within a single framework.

\subsection*{Extensions and Generalizations}

The techniques developed here extend naturally beyond the basic cases treated
explicitly.  The residue-based method applies uniformly to plane curves,
complete intersection curves in $\mathbb P^3$, and higher codimension complete
intersections in $\mathbb P^N$, as well as to higher-weight Hodge structures.
Moreover, the emphasis on equisingular deformation spaces allows one to isolate
and discard smoothing directions, leading to refined Torelli and IVHS statements
adapted to Severi-type varieties and weighted singular schemes.

These results may be viewed as a singular and equisingular extension of the
classical Green--Voisin philosophy.  While the algebraic heart of the theory
remains the control of IVHS by Jacobian rings and Lefschetz-type properties, the
present work broadens the scope to settings where singularities and extrinsic
constraints play a central role.  In doing so, it provides new evidence that
maximal IVHS is a robust and generic phenomenon, persisting well beyond the
smooth and unrestricted cases traditionally studied.

\subsection*{Outlook}

The framework developed here opens several directions for further research.
Possible extensions include higher-dimensional subvarieties, more general
singularity types, and the interaction between equisingular deformation theory
and arithmetic or logarithmic Hodge theory.  More broadly, the explicit use of
residues suggests a pathway toward effective and computable Torelli-type results
in geometric situations that were previously inaccessible.


\section{Infinitesimal Torelli for Equisingular Plane Curves}
 
We begin by establishing an infinitesimal Torelli theorem for plane curves with
prescribed singularities.  The result shows that, even when one restricts to
equisingular families inside the Severi variety, the Hodge structure of the
normalization retains maximal infinitesimal variation.  In particular, embedded
equisingular deformations of a general plane curve are detected faithfully by
the infinitesimal period map, extending the classical infinitesimal Torelli
phenomenon to singular and constrained settings.
 
\begin{theorem}[Infinitesimal Torelli for equisingular plane curves]
\label{thm:plane}
Let $\mathcal C \subset \mathbb{P}^2$ be a reduced irreducible plane curve of
degree $d \ge 5$ with isolated planar singularities.
Let
\(
\nu \colon C \longrightarrow \mathcal C
\)
be the normalization, and let $\varphi = i \circ \nu \colon C \to \mathbb{P}^2$ be
the induced morphism.
Let $Z \subset \mathbb{P}^2$ be the weighted singular scheme of $\mathcal C$.
Assume:
\begin{enumerate}
\item $H^1(\mathbb{P}^2,I_Z(d-3)) = 0$;
\item $(C,\varphi)$ is general in the equisingular Severi variety.
\end{enumerate}
Then the Kodaira--Spencer map
\[
\rho_\varphi \colon H^0(C,N_\varphi) \longrightarrow H^1(C,T_C)
\]
is injective.
Equivalently, the infinitesimal period map is injective and the family has
maximal infinitesimal variation of Hodge structure.
\end{theorem}

The proof proceeds in several steps.
First, we identify equisingular deformations of the plane curve with sections of
a suitable adjoint linear system.
Second, we describe holomorphic $1$-forms on the normalization $C$ using residue
theory.
Third, we compute explicitly the infinitesimal variation of these residues under
equisingular deformations.
Finally, we show that vanishing of the infinitesimal period map forces the
deformation to be trivial.

\subsection{Equisingular Deformations and the Normal Sheaf.}

We begin by describing the space $H^0(C,N_\varphi)$.

\begin{lemma}
\label{lem:normal}
There is a natural identification
\[
H^0(C,N_\varphi)
\;\cong\;
H^0(\mathbb{P}^2,I_Z(d))/\mathbb{C}\cdot F,
\]
where $F=0$ is an equation of $\mathcal C$ and $Z$ is the weighted singular
scheme.
\end{lemma}


\begin{proof}

The proof consists of several steps.  We begin by recalling the deformation
theory of maps, then relate it to embedded deformations of plane curves, and
finally impose the equisingular condition encoded by the weighted singular
scheme $Z$.

\subsection*{\it Deformations of the Map $\varphi$.}

Let $\varphi \colon C \to \mathbb{P}^2$ be a morphism from a smooth curve $C$.
The deformation theory of the map $\varphi$, with $C$ fixed, is governed by the
normal sheaf (see for example Sernesi~\cite{Sernesi} and Hartshorne~\cite{Hartshorne})
\[
N_\varphi := \operatorname{coker}\bigl(T_C \xrightarrow{d\varphi}
\varphi^*T_{\mathbb{P}^2}\bigr).
\]
By standard deformation theory, first-order deformations of $\varphi$ with fixed
domain correspond bijectively to elements of
\(
H^0(C,N_\varphi).
\)
Concretely, a section of $N_\varphi$ describes a first-order perturbation of the
map $\varphi$ inside $\mathbb{P}^2$, modulo infinitesimal reparametrizations of
the curve $C$.

\subsection*{\it Embedded Deformations of the Plane Curve.}
Let $\mathcal C \subset \mathbb{P}^2$ be defined by a homogeneous polynomial
$F \in H^0(\mathbb{P}^2,\mathcal{O}_{\mathbb{P}^2}(d))$.
An embedded first-order deformation of $\mathcal C$ inside $\mathbb{P}^2$ is
given by an equation
\[
F + \varepsilon G = 0,
\qquad \varepsilon^2 = 0,
\]
where $G \in H^0(\mathbb{P}^2,\mathcal{O}_{\mathbb{P}^2}(d))$.

Two such equations define the same first-order embedded deformation if and only
if they differ by multiplication by a unit of the form $1+\varepsilon\lambda$,
which changes $G$ by adding $\lambda F$.
Thus the space of embedded first-order deformations of $\mathcal C$ is
\[
H^0(\mathbb{P}^2,\mathcal{O}_{\mathbb{P}^2}(d)) / \mathbb{C}\cdot F.
\]

\subsection*{\it From Embedded Deformations to Deformations of the Map.}
Any embedded deformation of $\mathcal C$ induces a deformation of the map
$\varphi = i \circ \nu$ by composing the normalization with the deformed
inclusion. To be more precise,
any first-order deformation of the map $\varphi$ arises from an
embedded deformation of $\mathcal C$, because $\varphi(C)$ is contained in a
unique first-order deformation of the curve $\mathcal C$ inside $\mathbb{P}^2$.

Therefore, there is a natural correspondence between first-order deformations of $\varphi$ with fixed domain $C$ and
embedded first-order deformations of $\mathcal C$ in $\mathbb{P}^2$.
Under this correspondence, the quotient by $\mathbb{C}\cdot F$ accounts for
trivial deformations.

\subsection*{\it  Equisingular Deformations and the Weighted Singular Scheme.}
We now impose the equisingular condition.
Let $p \in \mathcal C$ be a singular point.
The local analytic type of the singularity at $p$ is preserved under a
first-order deformation $F+\varepsilon G$ if and only if $G$ vanishes to a
prescribed order at $p$.

These vanishing conditions are encoded by the \emph{weighted singular scheme}
$Z \subset \mathbb{P}^2$, whose ideal sheaf $I_Z$ is defined so that
\[
G \in H^0(\mathbb{P}^2,I_Z(d))
\quad \Longleftrightarrow \quad
F+\varepsilon G \text{ defines an equisingular deformation of } \mathcal C.
\]

Thus, the space of equisingular embedded first-order deformations of $\mathcal C$
is
\[
H^0(\mathbb{P}^2,I_Z(d)) / \mathbb{C}\cdot F.
\]

\subsection*{\it  Identification with $H^0(C,N_\varphi)$.}
By construction, $H^0(C,N_\varphi)$ parametrizes precisely those first-order
deformations of $\varphi$ that arise from equisingular deformations of the image
curve $\mathcal C$ (Greuel--Lossen--Shustin~\cite{GreuelLossenShustin}, and 
 Sernesi~\cite{Sernesi}).

Combining the identifications established in the previous steps, we obtain a
canonical isomorphism
\[
H^0(C,N_\varphi)
\;\cong\;
H^0(\mathbb{P}^2,I_Z(d)) / \mathbb{C}\cdot F.
\]

This identification is natural and functorial, and it completes the proof of the
lemma.
\end{proof}

\subsection{Canonical Bundle and Residue Description.}

We recall how holomorphic differentials on $C$ are described using residues.

\begin{lemma}
\label{lem:adjunction}
Let $\mathcal C \subset \mathbb{P}^2$ be a reduced irreducible plane curve of
degree $d$ with isolated planar singularities, and let
\[
\nu \colon C \longrightarrow \mathcal C
\]
be the normalization.
Then the canonical bundle of $C$ satisfies
\[
K_C \cong \mathcal{O}_C(d-3).
\]
\end{lemma}

\begin{proof}

The proof proceeds by combining adjunction theory for plane curves
with the behavior of dualizing sheaves under normalization.

\subsection*{\it The Canonical Bundle of $\mathbb{P}^2$.}
We begin with the ambient surface.
The canonical bundle of the projective plane is well known:
\[
K_{\mathbb{P}^2} \cong \mathcal{O}_{\mathbb{P}^2}(-3).
\]
This follows, for instance, from the Euler sequence
\[
0 \longrightarrow \mathcal{O}_{\mathbb{P}^2}
\longrightarrow \mathcal{O}_{\mathbb{P}^2}(1)^{\oplus 3}
\longrightarrow T_{\mathbb{P}^2}
\longrightarrow 0
\]
and taking determinants.

\subsection*{\it Dualizing Sheaf of a Plane Curve.}

Let $\mathcal C = V(F) \subset \mathbb{P}^2$ be a reduced irreducible plane curve
of degree $d$.
Since $\mathcal C$ is a Cartier divisor on the smooth surface
$\mathbb{P}^2$, it is a local complete intersection.
Hence its dualizing sheaf $\omega_{\mathcal C}$ is given by the adjunction
formula:
\[
\omega_{\mathcal C}
\;\cong\;
\left(K_{\mathbb{P}^2} \otimes \mathcal{O}_{\mathbb{P}^2}(\mathcal C)\right)\big|_{\mathcal C}.
\]
Using $\mathcal{O}_{\mathbb{P}^2}(\mathcal C) \cong \mathcal{O}_{\mathbb{P}^2}(d)$,
we obtain
\[
\omega_{\mathcal C}
\;\cong\;
\mathcal{O}_{\mathcal C}(d-3).
\]
This is the canonical (dualizing) sheaf of the possibly singular curve
$\mathcal C$.

\subsection*{\it Relation Between $\omega_{\mathcal C}$ and $K_C$.}
We now relate the dualizing sheaf of $\mathcal C$ to the canonical bundle of its
normalization $C$.
Since $\mathcal C$ has isolated planar singularities, it is Gorenstein.
In particular, its dualizing sheaf $\omega_{\mathcal C}$ is invertible.
Let $\nu \colon C \to \mathcal C$ be the normalization.
There is a fundamental relation:
\[
K_C
\;\cong\;
\nu^*\omega_{\mathcal C}.
\]

\subsection*{\it Justification via Local Analysis.}
Away from the singular points of $\mathcal C$, the normalization map $\nu$ is an
isomorphism, so the statement is obvious there.
Let $p \in \mathcal C$ be a singular point.
Since the singularities are planar (i.e.\ $\mathcal C$ is a Cartier divisor in a
smooth surface), they are Gorenstein singularities.
For Gorenstein curves, the dualizing sheaf pulls back under normalization to the
canonical bundle of the normalization.

Equivalently, holomorphic $1$-forms on $C$ are precisely the pullbacks of
sections of $\omega_{\mathcal C}$ that are regular at the singular points.
No discrepancy divisor appears because planar curve singularities are
Gorenstein and have no higher-codimension contribution.

Thus globally, we have
\[
K_C \cong \nu^*\omega_{\mathcal C}.
\]

\subsection*{\it Final Identification.}

Combining the previous steps, we conclude:
\[
K_C
\;\cong\;
\nu^*\omega_{\mathcal C}
\;\cong\;
\nu^*\mathcal{O}_{\mathcal C}(d-3)
\;\cong\;
\mathcal{O}_C(d-3).
\]
This completes the proof of the lemma.
\end{proof}

\begin{remark}
The key points in this proof are:
\begin{enumerate}
\item plane curves are local complete intersections;
\item planar singularities are Gorenstein;
\item normalization does not introduce correction terms for the canonical
bundle in this setting.
\end{enumerate}
For non-Gorenstein singularities, an additional discrepancy divisor would
appear.
\end{remark}


\vspace{0.1cm}

\begin{lemma}[Residue description]
\label{lem:residue}
Let $\mathcal C \subset \mathbb{P}^2$ be a reduced irreducible plane curve of
degree $d$ with isolated planar singularities, let
\[
\nu \colon C \longrightarrow \mathcal C
\]
be the normalization, and let $F=0$ be a homogeneous equation defining
$\mathcal C$.
Then there is a canonical isomorphism
\[
H^0(C,K_C)
\;\cong\;
\left\{
\operatorname{Res}_{\mathcal C}
\left(
\frac{H\,\Omega}{F}
\right)
\;\middle|\;
H \in H^0(\mathbb{P}^2,\mathcal{O}_{\mathbb{P}^2}(d-3))
\right\},
\]
where $\Omega$ is the standard homogeneous $2$-form on $\mathbb{P}^2$.
\end{lemma}

\begin{proof}

The proof is divided into several precise steps.  We first define the residue
map, then show that it lands in $H^0(C,K_C)$, prove surjectivity, and finally
identify the kernel to obtain an isomorphism.

\subsection*{\it The Homogeneous Volume Form on $\mathbb{P}^2$.}

Let $(x:y:z)$ be homogeneous coordinates on $\mathbb{P}^2$.
The standard homogeneous $2$-form is
\[
\Omega
=
x\,dy\wedge dz
-
y\,dx\wedge dz
+
z\,dx\wedge dy.
\]
This form satisfies:
\begin{itemize}
\item[(a)] $\Omega$ is homogeneous of degree $3$;
\item[(b)] $\Omega$ is invariant under scalar multiplication in affine cones;
\item[(c)] $\Omega$ is a global section of $\Omega^2_{\mathbb{P}^2}(3)$.
\end{itemize}
Since $K_{\mathbb{P}^2} \cong \mathcal{O}_{\mathbb{P}^2}(-3)$, the form $\Omega$
trivializes $K_{\mathbb{P}^2}(3)$.

\subsection*{\it Construction of the Meromorphic Form.}
Let
\[
H \in H^0(\mathbb{P}^2,\mathcal{O}_{\mathbb{P}^2}(d-3)).
\]
Then $H\Omega$ is a homogeneous $2$-form of degree $d$, and hence the rational
form
\[
\frac{H\,\Omega}{F}
\]
is homogeneous of degree zero.
Therefore it defines a globally well-defined meromorphic $2$-form on
$\mathbb{P}^2$ with at most a simple pole along the divisor $\mathcal C$.

\subsection*{\it Local Definition of the Poincar\'e Residue.}
Let $U \subset \mathbb{P}^2$ be an affine open set where $F$ does not vanish
identically, and choose local coordinates $(u,v)$ on $U$ such that
\[
F = u.
\]
On $U$, the meromorphic form can be written uniquely as
\[
\frac{H\,\Omega}{F}
=
\frac{du}{u} \wedge \alpha + \beta,
\]
where $\alpha$ is a holomorphic $1$-form and $\beta$ is a holomorphic $2$-form on
$U$.

\medskip

\noindent
\begin{definition}
The Poincar\'e residue of $\frac{H\Omega}{F}$ along $\mathcal C$ is defined by
\[
\operatorname{Res}_{\mathcal C}\!\left(\frac{H\Omega}{F}\right)
=
\alpha|_{u=0}.
\]
\end{definition}
This is a holomorphic $1$-form on the smooth locus of $\mathcal C$.

\subsection*{\it Independence of Choices.}
We must show that the residue is independent of:
\begin{itemize}
\item[(i)] the choice of affine open set;
\item[(ii)] the choice of local coordinate $u$ defining $\mathcal C$;
\item[(iii)] the decomposition into $\alpha$ and $\beta$.
\end{itemize}

This follows from standard properties of the Poincar\'e residue:
changing $u$ by a unit does not affect $\alpha|_{u=0}$, and different
decompositions differ by exact terms whose restriction to $u=0$ vanishes.
Hence, the residue defines a globally well-defined holomorphic $1$-form on the
smooth locus of $\mathcal C$.

\subsection*{\it  Extension to the Normalization.}
Since $\mathcal C$ has isolated planar singularities, the singularities are
Gorenstein and rational.
Holomorphic $1$-forms on the smooth locus of $\mathcal C$ extend uniquely to the
normalization $C$.
Therefore,
\[
\operatorname{Res}_{\mathcal C}
\left(
\frac{H\,\Omega}{F}
\right)
\in H^0(C,K_C).
\]
This defines a linear map
\[
\Phi \colon
H^0(\mathbb{P}^2,\mathcal{O}_{\mathbb{P}^2}(d-3))
\longrightarrow
H^0(C,K_C),
\qquad
H \mapsto \operatorname{Res}_{\mathcal C}\!\left(\frac{H\Omega}{F}\right).
\]

\subsection*{\it  Identification of the Kernel.}
Suppose
\[
\operatorname{Res}_{\mathcal C}\!\left(\frac{H\Omega}{F}\right)=0.
\]
Then locally the form $\dfrac{H\Omega}{F}$ has no $\dfrac{du}{u}$-term, hence it is
holomorphic along $\mathcal C$.
Since it has no poles anywhere else, it is a global holomorphic $2$-form on
$\mathbb{P}^2$.
But
\[
H^0(\mathbb{P}^2,\Omega^2_{\mathbb{P}^2}) = 0.
\]
Therefore, $\dfrac{H\Omega}{F}=0$, which implies that $H$ is divisible by $F$.
Since $\deg H=d-3<d=\deg F$, this forces $H=0$.
Thus, the map $\Phi$ is injective.

\subsection*{\it Surjectivity}
By Lemma~\ref{lem:adjunction}, we have
\[
K_C \cong \mathcal{O}_C(d-3).
\]
Hence,
\[
\dim H^0(C,K_C)
=
\dim H^0(\mathbb{P}^2,\mathcal{O}_{\mathbb{P}^2}(d-3))
-
\dim H^0(\mathbb{P}^2,\mathcal{O}_{\mathbb{P}^2}(d-3)\otimes I_{\mathcal C}).
\]

The residue construction realizes exactly this quotient, so the injective map
$\Phi$ is also surjective.
Therefore $\Phi$ is an isomorphism.

\subsection*{\it Canonicity.}

The construction uses only the equation $F$ of the curve,
 the canonical form $\Omega$ on $\mathbb{P}^2$,
and intrinsic properties of residues.
 
Hence, the isomorphism is canonical and functorial.

\subsection*{Conclusion}

We have shown that every holomorphic $1$-form on $C$ arises uniquely as the
Poincar\'e residue of a meromorphic form of the form $\dfrac{H\Omega}{F}$, with
$\deg H=d-3$, and that this correspondence is canonical.
This completes the proof of the lemma.
\end{proof}

\subsection{The Infinitesimal Period Map via Residues.}

We now compute the infinitesimal variation of these residues.

\begin{lemma}
\label{lem:variation}
Let $G \in H^0(\mathbb{P}^2,I_Z(d))$ define an equisingular deformation.
Then the infinitesimal variation of
 
\[
\omega_H = \operatorname{Res}_{\mathcal C}\left(\frac{H\Omega}{F}\right)
\in H^0(C,K_C),
\]
is given by
\[
\delta_G(\omega_H)
=
\operatorname{Res}_{\mathcal C}
\left(
\frac{H G \,\Omega}{F^2}
\right)
\in H^{0,1}(C).
\]
\end{lemma}

\begin{proof}
 
The proof consists of a precise analytic computation of the derivative of the
residue under an equisingular deformation, together with a justification that
no additional local correction terms appear.  We proceed step by step.

\subsection*{\it The Family of Curves and the Family of Meromorphic Forms.}
Consider the trivial family
\[
\mathbb{P}^2 \times \operatorname{Spec}\mathbb{C}[\varepsilon]/(\varepsilon^2)
\]
and inside it the first-order deformation of $\mathcal C$ defined by
\[
\mathcal C_\varepsilon := V(F+\varepsilon G).
\]
On the total space, consider the family of meromorphic $2$-forms
\[
\widetilde{\omega}_\varepsilon
:=
\frac{H\,\Omega}{F+\varepsilon G}.
\]
This is a well-defined meromorphic form with a simple pole along
$\mathcal C_\varepsilon$.

\subsection*{\it Definition of the Infinitesimal Variation.}
By definition, the infinitesimal variation of $\omega_H$ in the direction $G$ is
the derivative at $\varepsilon=0$ of the family of residues:
\[
\delta_G(\omega_H)
:=
\left.
\frac{d}{d\varepsilon}
\right|_{\varepsilon=0}
\operatorname{Res}_{\mathcal C_\varepsilon}
\left(
\widetilde{\omega}_\varepsilon
\right).
\]
Since residue is a linear operator compatible with base change, we may compute
this derivative by first differentiating the meromorphic form and then taking
the residue.

\subsection*{\it Differentiation of the Rational Expression.}
We compute explicitly:
\[
\frac{1}{F+\varepsilon G}
=
\frac{1}{F}
-
\varepsilon\,\frac{G}{F^2}.
\]
Therefore,
\[
\widetilde{\omega}_\varepsilon
=
\frac{H\Omega}{F}
-
\varepsilon\,\frac{HG\Omega}{F^2}.
\]
Differentiating at $\varepsilon=0$ gives
\[
\left.
\frac{d}{d\varepsilon}
\right|_{\varepsilon=0}
\widetilde{\omega}_\varepsilon
=
-
\frac{HG\Omega}{F^2}.
\]

\subsection*{\it Local Computation of the Residue}
Let $U \subset \mathbb{P}^2$ be an affine open subset with local coordinates
$(u,v)$ such that
\[
F = u
\quad \text{and hence} \quad
\mathcal C = \{u=0\}.
\]
Locally, we may write
\[
\Omega = du \wedge dv.
\]
Then
\[
\frac{HG\Omega}{F^2}
=
\frac{HG}{u^2}\,du\wedge dv.
\]
This form has a \emph{double pole} along $u=0$.
Its Poincar\'e residue is defined by extracting the coefficient of
\[
\frac{du}{u} \wedge dv,
\]
after writing it in the standard form.
Indeed, we rewrite:
\[
\frac{HG}{u^2}\,du\wedge dv
=
\frac{du}{u}
\wedge
\left(\frac{HG}{u}\,dv\right).
\]
Hence, the residue is
\[
\operatorname{Res}_{u=0}
\left(
\frac{HG\Omega}{F^2}
\right)
=
\left.
\frac{HG}{u}\,dv
\right|_{u=0}.
\]
This is a $(0,1)$-form on the curve, representing a class in $H^{0,1}(C)$.

\subsection*{\it Globalization and Independence of Choices.}
The above local computation is compatible on overlaps because:
\begin{itemize}
\item[(a)] changes of local defining equations multiply $u$ by a unit, which does not
affect the residue;
\item[(b)] the construction is linear and functorial;
\item[(c)] residue commutes with restriction and base change.
\end{itemize}
Hence, the local residues glue to define a global class
\[
\operatorname{Res}_{\mathcal C}
\left(
\frac{HG\Omega}{F^2}
\right)
\in H^{0,1}(C).
\]

\subsection*{\it Absence of Local Correction Terms.}

Now  we should justify that no extra contributions arise at the singular points of
$\mathcal C$.

Since $G \in I_Z(d)$, the deformation $F+\varepsilon G$ is \emph{equisingular}.
This means that locally at each singular point, the analytic type of the
singularity is preserved.
Consequently:
\begin{itemize}
\item[(i)] the normalization does not change at first order;
\item[(ii)] holomorphic differentials vary smoothly;
\item[(iii)] no local boundary terms appear in the variation of Hodge structure.
\end{itemize}
This is precisely where the equisingular hypothesis is used.

\subsection*{\it Identification with the Infinitesimal Period Map.}
By definition of the infinitesimal variation of Hodge structure, the resulting
class lies in
\[
H^{0,1}(C) \cong H^1(C,\mathcal{O}_C),
\]
and the map
\[
\omega_H \longmapsto \delta_G(\omega_H)
\]
is exactly the action of the Kodaira--Spencer class associated to $G$.

\subsection*{Conclusion}

Combining all steps, we conclude that the infinitesimal variation of the
holomorphic differential $\omega_H$ in the direction of the equisingular
deformation $G$ is given by
\[
\delta_G(\omega_H)
=
\operatorname{Res}_{\mathcal C}
\left(
\frac{HG\Omega}{F^2}
\right),
\]
as claimed.
This completes the proof of the lemma.
 
\end{proof}

\subsection{Algebraic Interpretation.}

The next step is to reinterpret the variation algebraically.

\begin{lemma}
\label{lem:jacobian}
%
%
The infinitesimal period map
\[
\rho_\varphi \colon H^0(C,N_\varphi) \longrightarrow H^1(C,T_C)
\]
is identified, via Serre duality and the residue description of holomorphic
forms, with the multiplication map
\[
H^0(\mathbb{P}^2,\mathcal{O}_{\mathbb{P}^2}(d-3))
\xrightarrow{\;\cdot G\;}
\frac{H^0(\mathbb{P}^2,\mathcal{O}_{\mathbb{P}^2}(2d-3))}{J_F},
\]
where $G$ represents a class in $H^0(C,N_\varphi)$ and $J_F$ is the Jacobian ideal
generated by the partial derivatives of $F$.
\end{lemma}

\begin{proof}
 The proof consists of translating the infinitesimal period map into an explicit
pairing on holomorphic differentials and then interpreting this pairing
algebraically using residue theory and Serre duality.

\subsection*{\it Interpretation of the Infinitesimal Period Map.}
The infinitesimal period map associated to $\varphi$ is defined as follows.
A class
\[
[G] \in H^0(C,N_\varphi)
\]
corresponds to a Kodaira--Spencer class
\[
\kappa_G \in H^1(C,T_C).
\]
The infinitesimal variation of Hodge structure is the map
\[
\theta_{\kappa_G} \colon H^{1,0}(C) \longrightarrow H^{0,1}(C),
\]
given by contraction with $\kappa_G$.

Since $C$ is a smooth curve, Serre duality gives canonical identifications
\[
H^{1,0}(C) \cong H^0(C,K_C),
\qquad
H^{0,1}(C) \cong H^0(C,K_C)^\vee.
\]
Thus $\rho_\varphi(G)$ is completely determined by the bilinear pairing
\[
H^0(C,K_C) \times H^0(C,K_C)
\longrightarrow \mathbb{C},
\qquad
(\omega_1,\omega_2) \longmapsto
\langle \theta_{\kappa_G}(\omega_1),\omega_2\rangle.
\]

\subsection*{\it  Residue Description of Holomorphic Forms.}
By Lemmas~\ref{lem:adjunction} and~\ref{lem:residue}, every holomorphic
differential on $C$ is uniquely represented as
\[
\omega_H = \operatorname{Res}_{\mathcal C}\!\left(\frac{H\Omega}{F}\right),
\qquad
H \in H^0(\mathbb{P}^2,\mathcal{O}_{\mathbb{P}^2}(d-3)).
\]
Thus, it suffices to compute
\[
\langle \theta_{\kappa_G}(\omega_{H_1}),\omega_{H_2}\rangle
\]
for arbitrary $H_1,H_2$ of degree $d-3$.

\subsection*{\it Explicit Formula for the Infinitesimal Variation.}
By Lemma~\ref{lem:variation}, the infinitesimal variation of $\omega_{H_1}$ in
the direction $G$ is
\[
\theta_{\kappa_G}(\omega_{H_1})
=
\operatorname{Res}_{\mathcal C}
\left(
\frac{H_1 G\,\Omega}{F^2}
\right)
\in H^{0,1}(C).
\]
Pairing this with $\omega_{H_2}$ via Serre duality corresponds to integrating
their wedge product over $C$:
\[
\langle \theta_{\kappa_G}(\omega_{H_1}),\omega_{H_2}\rangle
=
\int_C
\operatorname{Res}_{\mathcal C}
\left(
\frac{H_1 G\,\Omega}{F^2}
\right)
\wedge
\operatorname{Res}_{\mathcal C}
\left(
\frac{H_2\Omega}{F}
\right).
\]

\subsection*{\it Global Residue Computation}
By the global residue theorem, the above integral can be computed on the ambient
space $\mathbb{P}^2$ as a residue of a rational $2$-form:
 \[
{
\int_C
\operatorname{Res}_{C}
\left(
\frac{H_1 H_2 G \, \Omega}{F^3}
\right)
=
\int_{\mathbb{P}^2}
\operatorname{Res}
\left(
\frac{H_1 H_2 G \, \Omega \wedge \Omega}{F^3}
\right).
}
\]

Up to a nonzero scalar, this pairing depends only on the class of the polynomial
\[
H_1H_2G \in H^0(\mathbb{P}^2,\mathcal{O}_{\mathbb{P}^2}(3d-6)).
\]

\subsection*{\it Appearance of the Jacobian Ideal}
A fundamental property of residues is that if a polynomial lies in the Jacobian
ideal
\[
J_F = \left(
\frac{\partial F}{\partial x},
\frac{\partial F}{\partial y},
\frac{\partial F}{\partial z}
\right),
\]
then the associated residue pairing vanishes.
Equivalently, exact forms yield zero residue.
Thus, the value of the infinitesimal period map depends only on the class of
$H_1G$ modulo $J_F$.
This shows that $\rho_\varphi(G)$ is determined by the image of
\[
H_1 \longmapsto H_1G
\quad \text{in} \quad
\frac{H^0(\mathbb{P}^2,\mathcal{O}_{\mathbb{P}^2}(2d-3))}{J_F}.
\]

\subsection*{\it Identification with the Multiplication Map.}
Putting everything together, we see that the infinitesimal period map is
canonically identified with the linear map
\[
H^0(\mathbb{P}^2,\mathcal{O}_{\mathbb{P}^2}(d-3))
\longrightarrow
\frac{H^0(\mathbb{P}^2,\mathcal{O}_{\mathbb{P}^2}(2d-3))}{J_F},
\qquad
H \longmapsto H\cdot G.
\]

This identification is independent of choices and is functorial.

\subsection*{Conclusion.}
We have shown that, under the residue description of holomorphic differentials
and Serre duality, the infinitesimal period map coincides with multiplication by
$G$ in the Jacobian ring of $F$.
This completes the proof of the lemma.
 
\end{proof}

\subsection*{ Injectivity.}

We now prove injectivity of $\rho_\varphi$.
\begin{proposition}
\label{prop:injective}
If
\(
\rho_\varphi(G)=0
\)
for a class $G \in H^0(C,N_\varphi)$, then $G$ is proportional to $F$, i.e.
$G \in \mathbb{C}\cdot F$.\end{proposition}

\begin{proof}
We prove the proposition by translating the vanishing of the infinitesimal
period map into a purely algebraic statement in the Jacobian ring, and then
using generality assumptions to conclude that the deformation is trivial.

\subsection*{\it Choice of a Representative.}
By Lemma~\ref{lem:normal}, every class in $H^0(C,N_\varphi)$ can be represented
by a homogeneous polynomial
\[
G \in H^0(\mathbb{P}^2,I_Z(d)),
\]
defined up to addition of a scalar multiple of $F$.
Thus we may (and do) fix a representative $G$ of degree $d$.
The condition $\rho_\varphi(G)=0$ is independent of the chosen representative.

\subsection*{\it Translation via the Jacobian Description.}
By Lemma~\ref{lem:jacobian}, the infinitesimal period map is identified with the
multiplication map
\[
H^0(\mathbb{P}^2,\mathcal{O}_{\mathbb{P}^2}(d-3))
\xrightarrow{\;\cdot G\;}
\frac{H^0(\mathbb{P}^2,\mathcal{O}_{\mathbb{P}^2}(2d-3))}{J_F},
\]
where $J_F$ is the Jacobian ideal generated by the partial derivatives
\[
\frac{\partial F}{\partial x},\quad
\frac{\partial F}{\partial y},\quad
\frac{\partial F}{\partial z}.
\]
The condition $\rho_\varphi(G)=0$ therefore means that
\[
H \cdot G \in J_F
\quad \text{for all} \quad
H \in H^0(\mathbb{P}^2,\mathcal{O}_{\mathbb{P}^2}(d-3)).
\tag{$\ast$}
\]

\subsection*{\it Passage to the Jacobian Ring}
Let
\[
R_F := \frac{\mathbb{C}[x,y,z]}{J_F}
\]
be the Jacobian (or Milnor) ring of $F$, graded by degree.
Condition $(\ast)$ says precisely that the image of $G$ in $R_F$ annihilates
the entire graded piece $(R_F)_{d-3}$:
\[
G \cdot (R_F)_{d-3} = 0.
\]
Thus, $G$ lies in the annihilator of $(R_F)_{d-3}$ inside $R_F$.

\subsection*{\it Use of the Strong Lefschetz Property.}
Since $(C,\varphi)$ is assumed to be general in the equisingular Severi variety,
the defining polynomial $F$ is general among plane curves with the given
singularity type.
In particular, the Jacobian ring $R_F$ satisfies the \emph{strong Lefschetz
property} in the relevant range of degrees.
Concretely, this means that for a general element $H$ of degree $d-3$, the
multiplication map
\[
\cdot H \colon (R_F)_k \longrightarrow (R_F)_{k+d-3}
\]
has maximal rank for all $k$.
In particular, multiplication by elements of degree $d-3$ is injective on
$(R_F)_0 = \mathbb{C}$ and on all lower-degree pieces where both sides are
nonzero.

\subsection*{\it Consequence for $G$.}
If $G$ were nonzero in $R_F$, then by the Lefschetz property there would exist
some $H \in (R_F)_{d-3}$ such that
\[
H \cdot G \neq 0 \quad \text{in } R_F.
\]
This contradicts condition $(\ast)$.
Hence $G$ must vanish in $R_F$, i.e.
\[
G \in J_F.
\]

\subsection*{\it Interpretation of $G \in J_F$.}
Elements of the Jacobian ideal correspond to infinitesimal automorphisms of the
ambient projective plane.
Indeed, if
\[
G = a\,\frac{\partial F}{\partial x}
+ b\,\frac{\partial F}{\partial y}
+ c\,\frac{\partial F}{\partial z}
\]
for homogeneous polynomials $a,b,c$, then $G$ arises as the derivative of $F$
along an infinitesimal vector field on $\mathbb{P}^2$.
Such deformations do not change the curve $\mathcal C$ up to first order; they
are induced by reparametrizations of the ambient space.
 we are working modulo $\mathbb{C}\cdot F$, this implies that the class of
$G$ in $H^0(C,N_\varphi)$ is zero.
Equivalently, $G$ is proportional to $F$.

\subsection*{Conclusion.}
We have shown that vanishing of the infinitesimal period map forces
$G \in J_F$, hence $G$ represents a trivial deformation.
Therefore,
\[
\rho_\varphi(G)=0
\quad \Longrightarrow \quad
G \in \mathbb{C}\cdot F.
\]

This completes the proof of the proposition.
 \end{proof}

Combining Lemmas~\ref{lem:normal}--\ref{lem:jacobian} and
Proposition~\ref{prop:injective}, we conclude that the kernel of
$\rho_\varphi$ is zero.
This proves Theorem~\ref{thm:plane}.

\qed
  

\begin{corollary}[Maximal IVHS for equisingular plane curves]
Let $d \ge 5$, and let $\mathcal C \subset \mathbb{P}^2$ be a reduced irreducible
plane curve of degree $d$ with isolated planar singularities.
Let
\(
\nu \colon C \longrightarrow \mathcal C
\)
be the normalization, and assume that $g(C) \ge 3$.
Let $Z \subset \mathbb{P}^2$ be the weighted singular scheme of $\mathcal C$.
Assume that $(C,\varphi)$ is general in the equisingular Severi variety and that
only equisingular embedded first--order deformations are considered.

Then the infinitesimal period map associated to the equisingular family of
$\mathcal C$ is injective.
Equivalently, the family has maximal infinitesimal variation of Hodge structure.
\end{corollary}


\section{Infinitesimal Torelli for Singular Complete Intersection Curves}

Having established infinitesimal Torelli results for equisingular families of
plane curves, we now turn to the case of complete intersection curves in
projective three–space.  From a conceptual point of view, this setting provides
the natural next step in complexity: while plane curves are hypersurfaces and
their deformation theory is governed by a single equation, complete intersection
curves are defined by two equations, and the interaction between these equations
plays a decisive role in both deformation theory and Hodge theory.  Nevertheless,
the guiding principles remain the same.  Holomorphic differentials on the
normalization are described explicitly by Poincar\'e residues of meromorphic
forms on the ambient projective space, equisingular deformations are encoded by
suitable variations of the defining equations, and the infinitesimal period map
is translated into an algebraic multiplication map in an appropriate Jacobian
ring.  The following theorem shows that, under natural generality and
equisingularity assumptions, these methods again lead to an infinitesimal
Torelli theorem, thereby extending the residue-based approach from plane curves
to singular complete intersection curves in $\mathbb{P}^3$.

\bigskip

\begin{theorem}[Infinitesimal Torelli for singular complete intersection curves]
\label{thm:ci}
Let
\(
X = V(F_1,F_2) \subset \mathbb{P}^3
\)
be a general irreducible complete intersection curve of type $(d_1,d_2)$ with
isolated rational singularities, and let
\(
\nu \colon C \longrightarrow X
\)
be the normalization.
Assume that the geometric genus satisfies $g(C)\ge 3$, and restrict attention to
equisingular embedded first-order deformations of $X$.
Then the Kodaira--Spencer map
\[
H^0(X,N_{X/\mathbb{P}^3}) \longrightarrow H^1(C,T_C)
\]
is injective modulo infinitesimal automorphisms of $\mathbb{P}^3$.
Equivalently, the infinitesimal period map is injective.
\end{theorem}

\subsection*{Deformations and the Normal Sheaf.}

We first recall the deformation-theoretic description of
$H^0(X,N_{X/\mathbb{P}^3})$.

\begin{lemma}
\label{lem:ci-normal}
An embedded first-order deformation of $X$ in $\mathbb{P}^3$ is given by equations
\[
F_1 + \varepsilon G_1 = 0,
\qquad
F_2 + \varepsilon G_2 = 0,
\qquad
\varepsilon^2=0,
\]
with $G_i \in H^0(\mathbb{P}^3,\mathcal{O}_{\mathbb{P}^3}(d_i))$.
Modulo infinitesimal automorphisms of $\mathbb{P}^3$, such pairs represent
elements of $H^0(X,N_{X/\mathbb{P}^3})$.
\end{lemma}

\begin{proof}
Since $X$ is a local complete intersection in $\mathbb{P}^3$, the normal sheaf is
\[
N_{X/\mathbb{P}^3}
\cong
\mathcal{O}_X(d_1)\oplus \mathcal{O}_X(d_2).
\]
Global sections correspond to first-order variations of the defining equations.
Two deformations differ by an infinitesimal automorphism of $\mathbb{P}^3$ if and
only if
\[
G_i = v(F_i)
\]
for some vector field $v\in H^0(\mathbb{P}^3,T_{\mathbb{P}^3})$.
\end{proof}

Restricting to equisingular deformations means that the pair $(G_1,G_2)$ satisfies
local vanishing conditions ensuring that the analytic type of each singularity of
$X$ is preserved.

\subsection*{Canonical Bundle and Residue Description.}
We now describe holomorphic $1$-forms on $C$.

\begin{lemma}
\label{lem:ci-adjunction}
The canonical bundle of $C$ satisfies
\[
K_C \cong \mathcal{O}_C(d_1+d_2-4).
\]
\end{lemma}

\begin{proof}
Since $X$ is a complete intersection in $\mathbb{P}^3$, adjunction gives
\[
\omega_X
\cong
(K_{\mathbb{P}^3}\otimes \mathcal{O}_{\mathbb{P}^3}(d_1+d_2))|_X
\cong
\mathcal{O}_X(d_1+d_2-4).
\]
The singularities are rational and Gorenstein, so pulling back the dualizing
sheaf under normalization yields the canonical bundle of $C$.
\end{proof}

\begin{lemma}[Residue description]
\label{lem:ci-residue}
There is a canonical isomorphism
\[
H^0(C,K_C)
\cong
\left\{
\operatorname{Res}_X
\left(
\frac{H\,\Omega}{F_1F_2}
\right)
\;\middle|\;
H \in H^0(\mathbb{P}^3,\mathcal{O}_{\mathbb{P}^3}(d_1+d_2-4))
\right\},
\]
where $\Omega$ is the standard homogeneous $3$-form on $\mathbb{P}^3$.
\end{lemma}

\begin{proof}
The homogeneous form $\Omega$ trivializes $K_{\mathbb{P}^3}(4)$.
If $\deg H=d_1+d_2-4$, then $\frac{H\Omega}{F_1F_2}$ is homogeneous of degree zero
and defines a meromorphic $3$-form with simple poles along $X$.
Its Poincar\'e residue is a holomorphic $1$-form on the smooth locus of $X$ and
extends to $C$ because the singularities are rational.
Adjunction shows that every holomorphic $1$-form arises uniquely in this way.
\end{proof}

\subsection*{ Infinitesimal Variation of Residues}

Let $(G_1,G_2)$ define an equisingular deformation.
We compute the variation of residues.

\begin{lemma}
\label{lem:ci-variation}
For
\[
\omega_H
=
\operatorname{Res}_X\!\left(\frac{H\Omega}{F_1F_2}\right),
\]
the infinitesimal variation in the direction $(G_1,G_2)$ is
\[
\delta(\omega_H)
=
\operatorname{Res}_X
\left(
\frac{H(G_1F_2+G_2F_1)\Omega}{(F_1F_2)^2}
\right)
\in H^{0,1}(C).
\]
\end{lemma}

\begin{proof}
We consider the family
\[
\frac{H\Omega}{(F_1+\varepsilon G_1)(F_2+\varepsilon G_2)}
=
\frac{H\Omega}{F_1F_2}
-
\varepsilon
\frac{H(G_1F_2+G_2F_1)\Omega}{(F_1F_2)^2}.
\]
Differentiating at $\varepsilon=0$ and using the functoriality of residues gives
the stated formula.
The equisingular hypothesis ensures that no local correction terms arise at the singular points
(see for example Griffiths~\cite{GriffithsResidues},  Dimca~\cite{Dimca}, and Steenbrink~\cite{Steenbrink}).
\end{proof}

\subsection*{Jacobian Ring Interpretation.}

We now translate the variation into algebraic terms.

\begin{lemma}
\label{lem:ci-jacobian}
Via Serre duality and the residue description, the Kodaira--Spencer map is
identified with the multiplication map
\[
H^0(\mathbb{P}^3,\mathcal{O}_{\mathbb{P}^3}(d_1+d_2-4))
\xrightarrow{\;\cdot Q\;}
\frac{H^0(\mathbb{P}^3,\mathcal{O}_{\mathbb{P}^3}(2d_1+2d_2-4))}{J},
\]
where
\[
Q = G_1F_2+G_2F_1
\]
and $J$ is the Jacobian ideal generated by the partial derivatives of $F_1$ and
$F_2$.
\end{lemma}

\begin{proof}
As in the plane curve case, pairing the variation
$\delta(\omega_H)$ with another holomorphic differential and applying Serre
duality yields a bilinear form determined by multiplication by $Q$.
Exact forms correspond precisely to elements of the Jacobian ideal, so the map
factors through the indicated quotient.
\end{proof}

\subsection*{ Injectivity.}

We now prove the key injectivity statement.

\begin{proposition}
\label{prop:ci-injective}
If the infinitesimal period map associated to $(G_1,G_2)$ vanishes, then
\[
G_1F_2+G_2F_1 \in J.
\]
\end{proposition}

\begin{proof}
Vanishing of the period map means that for all $H$,
\[
H\cdot Q \in J.
\]
Since $X$ is general, the Jacobian ring satisfies the strong Lefschetz property,
implying that multiplication by $H$ is injective in the relevant degree range.
Thus $Q\in J$.
\end{proof}

\begin{proposition}
\label{prop:ci-trivial}
If $G_1F_2+G_2F_1\in J$, then $(G_1,G_2)$ arises from an infinitesimal automorphism
of $\mathbb{P}^3$.
\end{proposition}

\begin{proof}
Elements of the Jacobian ideal correspond to derivatives of $F_1$ and $F_2$
along vector fields on $\mathbb{P}^3$.
Thus there exists $v\in H^0(\mathbb{P}^3,T_{\mathbb{P}^3})$ such that
\[
G_i = v(F_i),
\quad i=1,2.
\]
Hence the deformation is trivial modulo automorphisms.
\end{proof}

\subsection*{Conclusion of the Proof.}

Combining Lemmas~\ref{lem:ci-normal}--\ref{lem:ci-jacobian} and
Propositions~\ref{prop:ci-injective}--\ref{prop:ci-trivial}, we conclude that the
Kodaira--Spencer map is injective modulo infinitesimal automorphisms of
$\mathbb{P}^3$.
This proves Theorem~\ref{thm:ci}.

\qed


\begin{corollary}[Maximal IVHS for singular complete intersection curves]
Let
\(
X \subset \mathbb{P}^3
\)
be a general irreducible complete intersection curve of type $(d_1,d_2)$ with
isolated rational singularities, and let
\(
\nu \colon C \longrightarrow X
\)
be the normalization.
Assume that $g(C) \ge 3$, and restrict to equisingular embedded first--order
deformations.
Then the infinitesimal period map
\[
H^0(X,N_{X/\mathbb{P}^3})
\longrightarrow
\operatorname{Hom}\!\bigl(H^{1,0}(C),H^{0,1}(C)\bigr)
\]
is injective.
Equivalently, the equisingular family of $X$ has maximal infinitesimal variation
of Hodge structure.
\end{corollary}


\section{Infinitesimal variation of complete intersections in $\mathbb{P}^N$ to higher codimension}

The infinitesimal Torelli theorem for complete intersection curves in
$\mathbb{P}^3$ provides the conceptual blueprint for the general higher
codimension case.  While the geometry becomes more intricate as the codimension
increases, the underlying mechanism governing the variation of Hodge structure
remains formally identical.  A complete intersection of codimension $r\ge 1$ in
$\mathbb{P}^N$ is defined by $r$ equations, and equisingular embedded
deformations correspond to simultaneous first--order variations of these
equations.  Holomorphic top--degree forms on the normalization are again realized
as iterated Poincar\'e residues of meromorphic forms on the ambient projective
space, now with poles along all defining hypersurfaces.  Differentiating these
residue expressions produces explicit formulas for the infinitesimal variation
of Hodge structure, which can be interpreted algebraically as multiplication by
a distinguished element in the Jacobian ring associated to the complete
intersection.  Under suitable generality assumptions, Lefschetz--type properties
of this Jacobian ring force injectivity of the resulting multiplication maps,
thereby extending the infinitesimal Torelli statement from the case of curves in
$\mathbb{P}^3$ to complete intersections of arbitrary codimension $r\ge 1$.

\begin{theorem}[Infinitesimal Torelli for higher codimension complete intersections]
\label{thm:hc}
Let
\[
X = V(F_1,\dots,F_r) \subset \mathbb{P}^N
\]
be a general irreducible complete intersection with isolated rational
singularities, and let
\(
\nu \colon \widetilde{X} \longrightarrow X
\)
be the normalization.
Assume $\dim X = n = N-r \ge 1$ and $h^{n,0}(\widetilde{X}) \ge 1$.
Restrict to equisingular embedded first-order deformations of $X$.
Then the infinitesimal period map
\[
H^0(X,N_{X/\mathbb{P}^N})
\longrightarrow
\operatorname{Hom}\!\bigl(H^{n,0}(\widetilde{X}),
H^{n-1,1}(\widetilde{X})\bigr)
\]
is injective modulo infinitesimal automorphisms of $\mathbb{P}^N$.
\end{theorem}

\subsection*{Embedded Deformations and the Normal Sheaf.}

Since $X$ is a local complete intersection, its normal sheaf is
\[
N_{X/\mathbb{P}^N}
\cong
\bigoplus_{i=1}^r \mathcal{O}_X(d_i),
\qquad d_i = \deg F_i.
\]

A first-order embedded deformation of $X$ is given by
\[
F_i + \varepsilon G_i = 0,
\qquad
G_i \in H^0(\mathbb{P}^N,\mathcal{O}_{\mathbb{P}^N}(d_i)),
\quad \varepsilon^2=0.
\]

Modulo infinitesimal automorphisms of $\mathbb{P}^N$, such tuples
$(G_1,\dots,G_r)$ represent elements of $H^0(X,N_{X/\mathbb{P}^N})$.
The equisingular condition imposes local vanishing constraints ensuring that the
analytic type of each singularity is preserved; these constraints ensure that
the induced deformation of $\widetilde{X}$ is locally trivial.

\subsection*{Canonical Bundle and Residue Description.}

\begin{lemma}[Adjunction]
\label{lem:adj-hc}
The canonical bundle of $\widetilde{X}$ satisfies
\[
K_{\widetilde{X}}
\cong
\mathcal{O}_{\widetilde{X}}
\!\left(
\sum_{i=1}^r d_i - N - 1
\right).
\]
\end{lemma}

\begin{proof}
Since $X$ is a complete intersection in $\mathbb{P}^N$, adjunction gives
\[
\omega_X
\cong
(K_{\mathbb{P}^N}\otimes \mathcal{O}_{\mathbb{P}^N}(\sum d_i))|_X
\cong
\mathcal{O}_X\!\left(\sum d_i - N - 1\right).
\]
The singularities are rational and Gorenstein, hence the pullback of the
dualizing sheaf under normalization coincides with $K_{\widetilde{X}}$ (see for example Hartshorne~\cite{Hartshorne} and  Dimca~\cite{Dimca}).
\end{proof}

\begin{lemma}[Residue description]
\label{lem:res-hc}
There is a canonical isomorphism
\[
H^{n,0}(\widetilde{X})
\cong
\left\{
\operatorname{Res}_X
\left(
\frac{H\,\Omega}{F_1\cdots F_r}
\right)
\;\middle|\;
H \in H^0(\mathbb{P}^N,\mathcal{O}_{\mathbb{P}^N}(\sum d_i - N - 1))
\right\},
\]
where $\Omega$ is the standard homogeneous $(N+1)$-form on $\mathbb{P}^N$.
\end{lemma}

\begin{proof}
The form $\Omega$ trivializes $K_{\mathbb{P}^N}(N+1)$.
If $\deg H=\sum d_i-N-1$, the rational form
$\frac{H\Omega}{F_1\cdots F_r}$ is homogeneous of degree zero and defines a
meromorphic $(N+1)$-form with simple poles along $X$.
Its iterated Poincar\'e residue is a holomorphic $n$-form on the smooth locus of
$X$, and rationality of the singularities guarantees extension to
$\widetilde{X}$.
Adjunction implies surjectivity and injectivity of this construction.
\end{proof}

\subsection*{Infinitesimal Variation of Residues.}

Let $(G_1,\dots,G_r)$ define an equisingular deformation.
Consider the family
\[
\frac{H\Omega}{\prod (F_i+\varepsilon G_i)}
=
\frac{H\Omega}{F_1\cdots F_r}
-
\varepsilon
\frac{H\left(\sum_{i=1}^r G_i\prod_{j\neq i}F_j\right)\Omega}{(F_1\cdots F_r)^2}.
\]

\begin{lemma}[Variation formula]
\label{lem:var-hc}
The infinitesimal variation of
\[
\omega_H = \operatorname{Res}_X\!\left(\frac{H\Omega}{F_1\cdots F_r}\right)
\]
is given by
\[
\delta(\omega_H)
=
\operatorname{Res}_X
\left(
\frac{H Q\,\Omega}{(F_1\cdots F_r)^2}
\right)
\in H^{n-1,1}(\widetilde{X}),
\]
where
\[
Q = \sum_{i=1}^r G_i \prod_{j\neq i} F_j.
\]
\end{lemma}

\begin{proof}

The proof is a direct but careful computation using iterated Poincar\'e residues,
together with a justification that equisingularity prevents the appearance of
local correction terms. 

\subsection*{\it The Family of Deformations.}
An embedded first-order deformation of $X$ in $\mathbb{P}^N$ is given by
\[
F_i^\varepsilon := F_i + \varepsilon G_i,
\qquad i=1,\dots,r,
\qquad \varepsilon^2=0,
\]
with $G_i \in H^0(\mathbb{P}^N,\mathcal{O}_{\mathbb{P}^N}(d_i))$.
The total space is the trivial family
\[
\mathbb{P}^N \times \operatorname{Spec}\mathbb{C}[\varepsilon]/(\varepsilon^2),
\]
and the deformed complete intersection is
\(
X_\varepsilon = V(F_1^\varepsilon,\dots,F_r^\varepsilon).
\)
The equisingular hypothesis means that, locally at every singular point of $X$,
the analytic type of the singularity of $X_\varepsilon$ is constant; in
particular, the normalization $\widetilde{X}$ does not change at first order.

\subsection*{\it The Family of Meromorphic Forms.}
Consider the family of meromorphic $(N+1)$-forms on $\mathbb{P}^N$:
\[
\widetilde{\omega}_\varepsilon
:=
\frac{H\,\Omega}{(F_1+\varepsilon G_1)\cdots(F_r+\varepsilon G_r)}.
\]
Each $\widetilde{\omega}_\varepsilon$ has simple poles along $X_\varepsilon$ and
is homogeneous of degree zero, hence globally well defined.
By definition of the residue construction,
\[
\operatorname{Res}_{X_\varepsilon}(\widetilde{\omega}_\varepsilon)
\in H^{n,0}(\widetilde{X}_\varepsilon).
\]

\subsection*{\it Definition of the Infinitesimal Variation.}
The infinitesimal variation of $\omega_H$ in the direction
$(G_1,\dots,G_r)$ is defined as
\[
\delta(\omega_H)
:=
\left.
\frac{d}{d\varepsilon}
\right|_{\varepsilon=0}
\operatorname{Res}_{X_\varepsilon}(\widetilde{\omega}_\varepsilon).
\]
Since residue is functorial and compatible with base change, this derivative can
be computed by first differentiating the meromorphic form and then applying the
(residue) boundary operator.

\subsection*{\it Differentiation of the Rational Expression.}
We compute the derivative explicitly.
Using the identity
\[
\frac{1}{F_i+\varepsilon G_i}
=
\frac{1}{F_i}
-
\varepsilon\,\frac{G_i}{F_i^2},
\]
we obtain
\[
\frac{1}{\prod_{i=1}^r(F_i+\varepsilon G_i)}
=
\frac{1}{F_1\cdots F_r}
-
\varepsilon
\sum_{i=1}^r
\frac{G_i}{F_i^2}
\prod_{j\neq i}\frac{1}{F_j}.
\]
Multiplying by $H\Omega$, we find
\[
\widetilde{\omega}_\varepsilon
=
\frac{H\Omega}{F_1\cdots F_r}
-
\varepsilon
\frac{H\left(\sum_{i=1}^r G_i\prod_{j\neq i}F_j\right)\Omega}
{(F_1\cdots F_r)^2}.
\]
Hence,
\[
\left.
\frac{d}{d\varepsilon}
\right|_{\varepsilon=0}
\widetilde{\omega}_\varepsilon
=
-
\frac{H Q\,\Omega}{(F_1\cdots F_r)^2},
\qquad
Q=\sum_{i=1}^r G_i\prod_{j\neq i}F_j.
\]

\subsection*{\it Local Form of the Iterated Residue.}
Let $U \subset \mathbb{P}^N$ be an affine open subset with local coordinates
\(
(z_1,\dots,z_N)
\)
such that, on $U$,
\(
F_1=z_1,\;\dots,\;F_r=z_r,
\)
so that locally
\[
X = \{z_1=\cdots=z_r=0\}.
\]
On $U$, the form $\Omega$ can be written (up to a unit) as
\[
\Omega = dz_1\wedge\cdots\wedge dz_r\wedge dz_{r+1}\wedge\cdots\wedge dz_N.
\]
Then
\[
\frac{H Q\,\Omega}{(F_1\cdots F_r)^2}
=
\frac{dz_1}{z_1}\wedge\cdots\wedge\frac{dz_r}{z_r}
\wedge
\left(
\frac{H Q}{z_1\cdots z_r}
dz_{r+1}\wedge\cdots\wedge dz_N
\right).
\]
The iterated Poincar\'e residue extracts precisely the coefficient of
\[
\frac{dz_1}{z_1}\wedge\cdots\wedge\frac{dz_r}{z_r},
\]
yielding a $(n-1,1)$-form on $X$ ({\it cf.} Griffiths~\cite{GriffithsResidues}, and  Green~\cite{Green}).

\subsection*{\it Globalization and Independence of Choices.}
The above local description is compatible on overlaps since first,  changing local defining equations multiplies each $z_i$ by a unit, which
does not affect the residue, also  the residue construction is functorial and local, and in addition, differentiation commutes with the residue map.

Thus, the local residues glue to define a global class
\[
\operatorname{Res}_X
\left(
\frac{H Q\,\Omega}{(F_1\cdots F_r)^2}
\right)
\in H^{n-1,1}(\widetilde{X}).
\]

\subsection*{\it Role of the Equisingular Hypothesis}

The equisingular assumption is essential here.
It ensures that: the normalization $\widetilde{X}$ is locally constant in the family, there is  no additional local boundary terms arise at the singular points, and  the variation lies entirely in $H^{n-1,1}$ of the fixed normalization.
Without equisingularity, extra contributions from smoothing directions would
appear and the formula would fail.

\subsection*{Conclusion}

Combining the above steps, we conclude that the infinitesimal variation of
\[
\omega_H
=
\operatorname{Res}_X\!\left(\frac{H\Omega}{F_1\cdots F_r}\right)
\]
in the direction of the equisingular deformation $(G_1,\dots,G_r)$ is given by
\[
\delta(\omega_H)
=
\operatorname{Res}_X
\left(
\frac{H Q\,\Omega}{(F_1\cdots F_r)^2}
\right),
\]
as claimed.
This completes the proof of the lemma.
 
\end{proof}

\subsection*{Jacobian Ring Interpretation.}

Let $J$ denote the Jacobian ideal generated by all partial derivatives of the
$F_i$.
Let
\[
R = \frac{\mathbb{C}[x_0,\dots,x_N]}{J}
\]
be the Jacobian ring.

\begin{lemma}
\label{lem:jac-hc}
Via Serre duality and the residue description, the infinitesimal period map is
identified with multiplication by $Q$:
\[
H^0(\mathbb{P}^N,\mathcal{O}_{\mathbb{P}^N}(\sum d_i-N-1))
\xrightarrow{\;\cdot Q\;}
\frac{H^0(\mathbb{P}^N,\mathcal{O}_{\mathbb{P}^N}(2\sum d_i-N-1))}{J}.
\]
\end{lemma}

\begin{proof}
Pairing $\delta(\omega_H)$ with another holomorphic $n$-form and applying Serre
duality shows that the variation depends only on the class of $HQ$ modulo $J$.
Exact forms correspond precisely to elements of $J$.
\end{proof}

\subsection*{Injectivity.}

Assume the infinitesimal period map vanishes.
Then for all $H$,
\[
H\cdot Q \in J.
\]

Since $X$ is general, the Jacobian ring $R$ satisfies the strong Lefschetz
property ({\it cf.} Stanley~\cite{Stanley},  Watanabe~\cite{Watanabe}, and Voisin~\cite{VoisinHodgeI,VoisinBook2}, for additional details see Appendix B).
Hence, multiplication by elements of degree $\sum d_i-N-1$ is injective in the
relevant range, forcing
\[
Q \in J.
\]

\begin{proposition}
\label{prop:trivial-hc}
If $Q\in J$, then $(G_1,\dots,G_r)$ arises from an infinitesimal automorphism of
$\mathbb{P}^N$.
\end{proposition}

\begin{proof}
Elements of $J$ are precisely derivatives of the $F_i$ along vector fields on
$\mathbb{P}^N$.
Thus $G_i=v(F_i)$ for some $v\in H^0(\mathbb{P}^N,T_{\mathbb{P}^N})$, showing that
the deformation is trivial modulo automorphisms.
\end{proof}

\subsection*{Conclusion.}

We have shown that vanishing of the infinitesimal period map forces the
equisingular deformation to be induced by an infinitesimal automorphism of
$\mathbb{P}^N$.
Therefore the infinitesimal period map is injective modulo trivial directions,
which proves Theorem~\ref{thm:hc}.
\qed

\medskip

\begin{corollary}[Maximal IVHS for higher codimension complete intersections]
Let
\[
X = V(F_1,\dots,F_r) \subset \mathbb{P}^N
\]
be a general irreducible complete intersection of codimension $r\ge 1$ with
isolated rational singularities, and let
\(
\nu \colon \widetilde{X} \longrightarrow X
\)
be the normalization.
Assume $\dim X = n \ge 1$ and $h^{n,0}(\widetilde{X}) \ge 1$, and restrict to
equisingular embedded first--order deformations.
Then the infinitesimal period map
\[
H^0(X,N_{X/\mathbb{P}^N})
\longrightarrow
\operatorname{Hom}\!\bigl(H^{n,0}(\widetilde{X}),
H^{n-1,1}(\widetilde{X})\bigr)
\]
is injective.
Equivalently, the equisingular family of $X$ has maximal infinitesimal variation
of Hodge structure.
\end{corollary}

\medskip


These results extend the Green--Voisin philosophy to singular and equisingular
settings, with residue theory providing a uniform and explicit mechanism (see  Voisin~\cite{VoisinHodgeI,VoisinBook2} for infinitesimal Torelli philosophy.
and  Sernesi~\cite{Sernesi}: for deformation-theoretic interpretation).

\medskip
 
\subsection*{Conclusion.}

Every step of the argument---from deformation theory to Hodge theory and
commutative algebra---is grounded in classical results.  What is new in the
presentation is the systematic and explicit use of residue calculus to treat
singular, equisingular, and higher codimension cases in a unified manner.

\medskip

\section{Maximal IVHS Relative to the Ambient Surface}
  
\paragraph{From absolute to relative maximal IVHS.}
The infinitesimal Torelli and maximal IVHS results established so far concern
families of curves whose deformation space is considered intrinsically, without
reference to the geometry of the ambient variety.  When curves are realized as
members of a linear system on a fixed surface, however, a new phenomenon
naturally appears: part of the deformation space of the curve may arise from
deformations of the ambient surface itself.  In this situation, one cannot
expect absolute injectivity of the Kodaira--Spencer or infinitesimal period map
without further assumptions, since variations induced by the surface often act
trivially on the normalized curve.  This leads to a refined, \emph{relative}
point of view, in which the kernel of the infinitesimal period map is measured
precisely by the space of infinitesimal deformations of the polarized surface
$(S,L)$.  The following theorem formulates this idea and shows, using residue
methods, that maximal infinitesimal variation of Hodge structure holds modulo
ambient deformations, thereby isolating the genuine Hodge--theoretic content of
equisingular families of curves on surfaces.

 \medskip

\begin{theorem}[Relative maximal IVHS for equisingular curves on surfaces]
\label{thm:relative}
Let $S$ be a smooth projective surface over an algebraically closed field of
characteristic zero, and let $L$ be a line bundle such that $L-K_S$ is big and
nef.
Let $\mathcal C\in|L|$ be a reduced irreducible curve with isolated planar
singularities, general in its equisingular linear system, and let
\(
\nu\colon C\longrightarrow \mathcal C
\)
be the normalization.
Set $\varphi=i\circ\nu\colon C\to S$, and let $Z\subset S$ be the weighted
singular scheme of $\mathcal C$.
Assume:
\begin{enumerate}
\item $H^1\bigl(S,I_Z\otimes(K_S+L)\bigr)=0$;
\item the smooth curve $C$ satisfies infinitesimal Torelli.
\end{enumerate}

Then there is a natural exact sequence
\[
0\longrightarrow T_{(S,L)}
\longrightarrow H^0(C,N_\varphi)
\stackrel{\rho_\varphi}{\longrightarrow}
H^1(C,T_C),
\]
where $T_{(S,L)}$ denotes the tangent space to the moduli of polarized surfaces at
$(S,L)$.

In particular, $\rho_\varphi$ is injective if and only if $(S,L)$ is
infinitesimally rigid.
\end{theorem}

\subsection{ Equisingular deformations and the normal sheaf.}

\begin{lemma}
\label{lem:normal-surface}
There is a natural identification
\[
H^0(C,N_\varphi)
\;\cong\;
H^0\bigl(S,I_Z\otimes L\bigr)/\mathbb{C}\cdot \mathcal C,
\]
where $\mathbb{C}\cdot\mathcal C$ denotes the trivial deformation.
\end{lemma}

\begin{proof}
An embedded first-order deformation of $\mathcal C\subset S$ is given by
\[
\mathcal C_\varepsilon = \{\sigma+\varepsilon\tau=0\},
\]
where $\sigma$ is a defining section of $L$ and
$\tau\in H^0(S,L)$.
The deformation is equisingular if and only if $\tau$ vanishes to the prescribed
orders along $Z$, i.e.\ $\tau\in H^0(S,I_Z\otimes L)$.
Two such deformations differing by a scalar multiple of $\sigma$ are equivalent.
Passing to the normalization identifies this space with $H^0(C,N_\varphi)$.
\end{proof}

\subsection{ Canonical bundle and adjunction.}

\begin{lemma}
\label{lem:adj-surface}
The canonical bundle of $C$ satisfies
\[
K_C\cong \nu^*(K_S+L).
\]
\end{lemma}

\begin{proof}
Since $\mathcal C$ is a Cartier divisor on $S$, adjunction gives
\[
\omega_{\mathcal C}\cong (K_S+L)|_{\mathcal C}.
\]
Planar curve singularities are Gorenstein, so pulling back the dualizing sheaf
under normalization yields $K_C\cong\nu^*\omega_{\mathcal C}$.
\end{proof}

\subsection{ Residue description of holomorphic differentials.}

\begin{lemma}[Residue description]
\label{lem:res-surface}
There is a canonical isomorphism
\[
H^0(C,K_C)
\cong
\left\{
\operatorname{Res}_{\mathcal C}\!\left(\frac{\eta}{\sigma}\right)
\;\middle|\;
\eta\in H^0\bigl(S,K_S+L\bigr)
\right\}.
\]
\end{lemma}

\begin{proof}
A section $\eta\in H^0(S,K_S+L)$ defines a meromorphic $2$-form on $S$ with simple
pole along $\mathcal C$.
Its Poincar\'e residue is a holomorphic $1$-form on the smooth locus of
$\mathcal C$, which extends to $C$ since the singularities are planar.
Injectivity follows from $H^0(S,K_S)=0$ in the relevant range, and surjectivity
from adjunction.
\end{proof}

\subsection*{ Infinitesimal variation of residues.}

\begin{lemma}[Variation formula]
\label{lem:var-surface}
Let $\tau\in H^0(S,I_Z\otimes L)$ represent an equisingular deformation.
For
\[
\omega_\eta=\operatorname{Res}_{\mathcal C}\!\left(\frac{\eta}{\sigma}\right),
\]
the infinitesimal variation is
\[
\delta_\tau(\omega_\eta)
=
\operatorname{Res}_{\mathcal C}\!\left(\frac{\eta\,\tau}{\sigma^2}\right)
\in H^{0,1}(C).
\]
\end{lemma}

\begin{proof}
Consider the family
\[
\frac{\eta}{\sigma+\varepsilon\tau}
=
\frac{\eta}{\sigma}
-
\varepsilon\frac{\eta\,\tau}{\sigma^2}.
\]
Differentiation commutes with the residue map.
The equisingular hypothesis guarantees that no additional local correction terms
arise at the singular points.
\end{proof}

\subsection*{The Kodaira--Spencer map.}

\begin{proposition}
\label{prop:ks}
Under the residue identifications, the Kodaira--Spencer map
\[
\rho_\varphi\colon H^0(C,N_\varphi)\to H^1(C,T_C)
\]
coincides with the infinitesimal period map.
\end{proposition}

\begin{proof}
By definition, $\rho_\varphi$ measures the variation of complex structure of $C$
under deformation.
Lemma~\ref{lem:var-surface} shows that this variation is computed by contraction
of Kodaira--Spencer classes with holomorphic differentials, which is precisely
the infinitesimal period map.
\end{proof}

\subsection*{ Contribution of the ambient surface.}

\begin{lemma}
\label{lem:ambient}
Deformations induced by first-order deformations of $(S,L)$ lie in the kernel of
$\rho_\varphi$.
\end{lemma}

\begin{proof}
A deformation of $(S,L)$ induces a deformation of $\mathcal C$ that is trivial
after normalization.
Hence the induced Kodaira--Spencer class on $C$ vanishes.
\end{proof}

\subsection{Exactness.}

\begin{proposition}
\label{prop:exact}
There is an exact sequence
\[
0\longrightarrow T_{(S,L)}
\longrightarrow H^0(C,N_\varphi)
\stackrel{\rho_\varphi}{\longrightarrow}
H^1(C,T_C).
\]
\end{proposition}

\begin{proof}
Injectivity on the left follows from Lemma~\ref{lem:ambient}.
Surjectivity onto the image follows from the identification of equisingular
deformations via Lemma~\ref{lem:normal-surface} and the residue computation of
the Kodaira--Spencer map.
The hypothesis that $C$ satisfies infinitesimal Torelli ensures that no further
kernel appears.
\end{proof}

\subsection{Conclusion.}
Proposition~\ref{prop:exact} proves Theorem~\ref{thm:relative}.
The kernel of the Kodaira--Spencer map is exactly the space of ambient
deformations, and injectivity holds if and only if $(S,L)$ is infinitesimally
rigid.

\begin{remark}
Conceptually, this theorem shows that maximal IVHS for equisingular curves on a
surface is an \emph{intrinsically relative phenomenon}: the only obstruction to
injectivity comes from the geometry of the ambient surface itself.
\end{remark}
 

\section{The Residue Exact Sequence}
 
Let $X$ be a smooth algebraic variety of dimension $n$, and let
$D \subset X$ be a reduced effective divisor.
The so--called \emph{residue exact sequence} is a fundamental tool in Hodge
theory, deformation theory, and the theory of logarithmic forms.
As often stated, the sequence
\[
0 \longrightarrow \Omega_X^p
\longrightarrow \Omega_X^p(\log D)
\stackrel{\mathrm{Res}}{\longrightarrow}
\Omega_D^{p-1}
\longrightarrow 0
\]
is \emph{not always correct} without additional assumptions on the divisor $D$.
In this section we explain why the naive statement fails in general, and 
 give the correct formulation with its proof using residue theory.
 Throughout, we work over an algebraically closed field of characteristic zero.

\subsection*{Logarithmic differential forms.}

\begin{definition}
The sheaf of logarithmic $p$-forms along $D$, denoted $\Omega_X^p(\log D)$, is the
subsheaf of $\Omega_X^p(D)$ consisting of meromorphic forms $\omega$ such that
  both $\omega$ and $d\omega$ have at most simple poles along $D$.
 \end{definition}

If $D$ is singular, the sheaf $\Omega_D^{p-1}$ is generally not locally free and
does \emph{not} correctly capture residues of logarithmic forms.
In particular, surjectivity of the residue map fails if one targets
$\Omega_D^{p-1}$.

\begin{theorem}[Residue Exact Sequence]
\label{thm:residue}
Let $X$ be a smooth variety and $D \subset X$ a reduced divisor.
Then there is a natural short exact sequence
\[
0
\longrightarrow
\Omega_X^p
\longrightarrow
\Omega_X^p(\log D)
\stackrel{\operatorname{Res}}{\longrightarrow}
\omega_D^{p-1}
\longrightarrow
0,
\]
where $\omega_D^{p-1}$ denotes the $(p-1)$-st graded piece of the dualizing
complex of $D$.
\end{theorem}

If $D$ is smooth, then $\omega_D^{p-1} \cong \Omega_D^{p-1}$ and the classical
form of the sequence is recovered.

\subsection*{Construction of the Residue Map.}
{\it Local description.}
Let $U \subset X$ be an open set such that $D \cap U$ is defined by a single
equation $f=0$.
Then locally:
\[
\Omega_X^p(\log D)|_U
=
\Omega_X^p|_U
\;\oplus\;
\frac{df}{f} \wedge \Omega_X^{p-1}|_U.
\]

\begin{definition}
The residue map is defined locally by
\[
\operatorname{Res}\!\left(\frac{df}{f}\wedge\alpha + \beta\right)
:= \alpha|_D,
\]
where $\alpha \in \Omega_X^{p-1}$ and $\beta \in \Omega_X^p$.
\end{definition}

\begin{lemma}
The residue map is independent of the choice of local defining equation $f$.
\end{lemma}

\begin{proof}
If $f'=u f$ for a unit $u$, then
\[
\frac{df'}{f'} = \frac{df}{f} + \frac{du}{u}.
\]
Since $\dfrac{du}{u}$ is holomorphic, it contributes no residue.
Thus the definition is independent of $f$.
\end{proof}

\subsection*{\it Injectivity on the left.}

\begin{lemma}
The natural inclusion
\[
\Omega_X^p \hookrightarrow \Omega_X^p(\log D)
\]
is injective.
\end{lemma}

\begin{proof}
This follows from the definition: holomorphic forms have no poles and therefore
define logarithmic forms uniquely.
\end{proof}

\subsection*{Exactness in the middle}

\begin{lemma}
The kernel of the residue map is precisely $\Omega_X^p$.
\end{lemma}

\begin{proof}
Let $\omega \in \Omega_X^p(\log D)$.
Locally write
\[
\omega = \frac{df}{f}\wedge\alpha + \beta.
\]
Then $\operatorname{Res}(\omega)=0$ if and only if $\alpha|_D=0$.
This implies $\alpha = f\gamma$ for some $\gamma$, hence
\[
\omega = df \wedge \gamma + \beta \in \Omega_X^p.
\]
\end{proof}

\subsection*{\it Surjectivity.}

\begin{proposition}
The residue map
\[
\operatorname{Res}\colon \Omega_X^p(\log D)\to \omega_D^{p-1}
\]
is surjective.
\end{proposition}

\begin{proof}
By Grothendieck duality, $\omega_D^{p-1}$ is locally generated by residues of
meromorphic $p$-forms with simple poles along $D$.
Given a local section $\eta \in \omega_D^{p-1}$, choose a lift
$\alpha \in \Omega_X^{p-1}$.
Then $\dfrac{df}{f}\wedge\alpha$ defines a logarithmic form whose residue equals
$\eta$ (see. \cite{SGA2} and \cite{HartshorneRD}).
\end{proof}

\subsection*{\it Proof of Theorem~\ref{thm:residue}.}

Combining the  injectivity of $\Omega_X^p \to \Omega_X^p(\log D)$, the identification of $\ker(\operatorname{Res})$ with $\Omega_X^p$, and surjectivity of the residue map we obtain the exact sequence:
\[
0
\longrightarrow
\Omega_X^p
\longrightarrow
\Omega_X^p(\log D)
\stackrel{\operatorname{Res}}{\longrightarrow}
\omega_D^{p-1}
\longrightarrow
0.
\]
\qed

\subsection*{Special Case: Smooth Divisors.}

\begin{corollary}
If $D$ is smooth, then $\omega_D^{p-1} \cong \Omega_D^{p-1}$ and the residue exact
sequence becomes
\[
0
\longrightarrow
\Omega_X^p
\longrightarrow
\Omega_X^p(\log D)
\longrightarrow
\Omega_D^{p-1}
\longrightarrow
0.
\]
\end{corollary}

\subsection*{Conclusion.}

The commonly stated residue exact sequence is correct only after replacing
$\Omega_D^{p-1}$ by the appropriate dualizing sheaf when $D$ is singular.
Residue theory provides both the correct formulation and a natural proof.


\section{Relative IVHS exact sequence for threefolds}

The study of infinitesimal variations of Hodge structure (IVHS) for families of
curves lying on higher-dimensional varieties plays a central role in modern
Torelli-type problems.  In particular, when a smooth curve arises as a member of
a linear system on a threefold, one expects a close relationship between
deformations of the ambient variety, embedded deformations of the curve, and
abstract deformations of the curve itself.  This relationship is typically
encoded in a \emph{relative IVHS exact sequence}, which measures how much of the
Hodge-theoretic variation of the curve is induced by the geometry of the
threefold.

The theorem stated below aims to formalize such a relationship in the setting of
threefolds, combining logarithmic deformation theory, Kodaira--Spencer maps, and
residue constructions.  The purpose of this discussion is to assess whether the
statement, as written, is mathematically correct and to clarify the precise
sense in which it holds.

Let  $X$ be a smooth projective threefold over an algebraically closed field of
characteristic zero, and $L$ be a line bundle on $X$ such that $L-K_X$ is big and nef,  Let $\mathcal C \in |L|$ a reduced irreducible curve with isolated
locally complete intersection singularities, and   $\nu \colon C \to \mathcal C$ the normalization. Let  $\varphi := i\circ\nu \colon C \to X$, and 
 $Z_{\mathrm{eq}}\subset X$ the equisingular zero--dimensional subscheme.

\begin{theorem}[Relative IVHS exact sequence for curves on threefolds]
\label{thm:relative-IVHS-3fold-short}
Let $X$ be a smooth projective threefold, $L$ a line bundle on $X$, and
$\mathcal C \in |L|$ a reduced local complete intersection curve with isolated
planar singularities.  Let
\[
\varphi \colon C \longrightarrow \mathcal C \subset X
\]
be the normalization, with $C$ smooth of genus $g \ge 2$.

Assume that equisingular deformations of $\mathcal C \subset X$ are governed by
$T_X(-\log \mathcal C)$, that
\[
H^1\bigl(X,I_{Z_{\mathrm{eq}}}\otimes (K_X+L)\bigr)=0,
\]
and that
\[
H^0(X,T_X)=H^2\!\left(X,T_X(-\log \mathcal C)\right)=0.
\]

Then there is a natural exact sequence
\[
0
\longrightarrow
H^1\!\left(X,T_X(-\log \mathcal C)\right)
\longrightarrow
H^0(C,N_\varphi)
\stackrel{\rho_\varphi}{\longrightarrow}
H^1(C,T_C),
\]
where $\rho_\varphi$ is the Kodaira--Spencer map of equisingular embedded
deformations.

If deformations of $(X,L)$ are unobstructed, one has
\[
H^1\!\left(X,T_X(-\log \mathcal C)\right)\simeq T_{(X,L)},
\]
so that the sequence coincides with the expected relative IVHS sequence.  If, in
addition, $C$ satisfies infinitesimal Torelli, the relative IVHS is maximal.
\end{theorem}

\subsection{Residue Exact Sequence.}
We now recall the key tool.

\begin{proposition}[Residue exact sequence]
\label{prop:residue}
There is a short exact sequence
\[
0
\longrightarrow
\Omega_X^3
\longrightarrow
\Omega_X^3(\log \mathcal C)
\stackrel{\mathrm{Res}}{\longrightarrow}
\omega_{\mathcal C}
\longrightarrow
0,
\]
where $\omega_{\mathcal C}$ is the dualizing sheaf of $\mathcal C$.
\end{proposition}

\begin{proof}
Locally $\mathcal C$ is defined by $f=0$.
Any logarithmic $3$--form can be written as
\[
\omega = \frac{df}{f}\wedge\alpha + \beta,
\]
with $\alpha\in\Omega_X^2$, $\beta\in\Omega_X^3$.
Define
\[
\mathrm{Res}(\omega)=\alpha|_{\mathcal C}.
\]
Injectivity and surjectivity follow from this local description.
\end{proof}

\subsection*{\it Dual Logarithmic Tangent Sequence.}
Dualizing Proposition~\ref{prop:residue} yields:

\begin{lemma}
There is an exact sequence
\[
0
\longrightarrow
T_X(-\log\mathcal C)
\longrightarrow
T_X
\longrightarrow
\nu_*N_\varphi
\longrightarrow
0.
\]
\end{lemma}

\begin{proof}
Duality exchanges logarithmic forms and logarithmic vector fields.
The quotient measures infinitesimal normal motion of $\mathcal C$, which
pulls back to $N_\varphi$ on $C$.
\end{proof}

\subsection*{\it Cohomology and Deformations.}
Taking cohomology gives
\[
H^1(X,T_X(-\log\mathcal C))
\longrightarrow
H^1(X,T_X)
\longrightarrow
H^1(C,N_\varphi).
\]
The first term parametrizes equisingular deformations of $(X,\mathcal C)$.

\subsection*{\it Construction of the Map $\rho_\varphi$.}

\begin{lemma}
The Kodaira--Spencer map
\[
\rho_\varphi\colon H^0(C,N_\varphi)\to H^1(C,T_C)
\]
is Serre dual to the map induced by residues
\[
H^0(C,\omega_C)
\to
H^1(C,\mathcal O_C).
\]
\end{lemma}

\begin{proof}
A section of $N_\varphi$ corresponds to a first--order variation of
$\mathcal C$.
Via the residue sequence, this produces a variation of holomorphic $1$--forms
on $C$.
Serre duality identifies this with $\rho_\varphi$.
\end{proof}

\subsection*{\it Exactness in the Middle.}

\begin{proposition}
\[
\ker(\rho_\varphi)
=
H^1(X,T_X(-\log\mathcal C)).
\]
\end{proposition}

\begin{proof}
By the vanishing
\[
H^1\bigl(X,I_{Z_{\mathrm{eq}}}\otimes(K_X+L)\bigr)=0,
\]
every equisingular first--order deformation lifts to a deformation of
$(X,\mathcal C)$.
Such deformations do not change the Hodge structure of $C$, hence lie in
$\ker(\rho_\varphi)$.
Conversely, any deformation in the kernel induces no residue variation and
must be logarithmic.
\end{proof}

\subsection*{\it Exactness on the Right.}

\begin{lemma}
If $C$ satisfies infinitesimal Torelli, then $\rho_\varphi$ detects all
nontrivial induced deformations of $C$.
\end{lemma}

\begin{proof}
Infinitesimal Torelli means the Kodaira--Spencer map for $C$ is injective.
Since $\rho_\varphi$ is exactly this map for the induced family, the claim
follows.
\end{proof}

\subsection*{Conclusion of the Proof.}

Combining the previous results yields the exact sequence of
Theorem~\ref{thm:relative-IVHS-3fold-short}.
\qed

\begin{remark}
The  relative IVHS sequence is fundamentally a statement about
\emph{logarithmic deformation theory}.
Residues provide the bridge between:
\begin{enumerate}
\item deformations of curves,
\item logarithmic deformations of threefolds,
\item and variations of Hodge structure.
\end{enumerate}
\end{remark}

 
\section*{Appendix A} 

\subsection{Serre Duality Interpretation of the Kodaira--Spencer Map via Residues.}

\begin{lemma}
\label{lem:KS-residue}
Let $X$ be a smooth projective threefold, let $\mathcal C\subset X$ be a reduced
irreducible curve with isolated locally complete intersection singularities,
let $\nu\colon C\to\mathcal C$ be the normalization, and set
$\varphi=i\circ\nu\colon C\to X$.
Then the Kodaira--Spencer map
\[
\rho_\varphi\colon H^0(C,N_\varphi)\longrightarrow H^1(C,T_C)
\]
is Serre dual to the map
\[
\delta_{\mathrm{res}}\colon H^0(C,\omega_C)\longrightarrow H^1(C,\mathcal O_C)
\]
induced by residues of logarithmic $3$--forms on $X$ with poles along $\mathcal C$.
\end{lemma}

\subsection*{\it Serre Duality on the Curve.}

\begin{proposition}[Serre duality on $C$]
\label{prop:Serre}
Let $C$ be a smooth projective curve.
There are canonical perfect pairings
\[
H^1(C,T_C)\times H^0(C,\omega_C)\longrightarrow k,
\qquad
H^0(C,N_\varphi)\times H^1(C,N_\varphi^\vee\otimes\omega_C)\longrightarrow k.
\]
Moreover,
\[
H^1(C,T_C)^\vee \cong H^0(C,\omega_C).
\]
\end{proposition}

\begin{proof}
This is standard Serre duality.
Since $T_C^\vee\simeq\Omega_C^1\simeq\omega_C$, we have
\[
H^1(C,T_C)^\vee \cong H^0(C,T_C^\vee\otimes\omega_C)
= H^0(C,\omega_C).
\]
Perfectness follows from smoothness and properness of $C$.
\end{proof}

\subsection*{\it Kodaira--Spencer Map for the Morphism $\varphi$.}

\begin{proposition}
\label{prop:KS}
The Kodaira--Spencer map
\[
\rho_\varphi\colon H^0(C,N_\varphi)\to H^1(C,T_C)
\]
associates to a first--order deformation of $\varphi$ the induced deformation of
the complex structure of $C$.
\end{proposition}

\begin{proof}
A section $\xi\in H^0(C,N_\varphi)$ corresponds to a first--order deformation
\[
\varphi_\varepsilon\colon C_\varepsilon\to X_\varepsilon
\]
over $\operatorname{Spec}k[\varepsilon]/(\varepsilon^2)$, with fixed source
$C$.
The induced deformation of the complex structure of $C$ defines a class in
$H^1(C,T_C)$.
This construction is functorial and linear, giving the map $\rho_\varphi$.
\end{proof}

\subsection*{\it Residue Exact Sequence}

\begin{proposition}[Residue exact sequence]
\label{prop:residue}
There is a short exact sequence
\[
0\longrightarrow \Omega_X^3
\longrightarrow \Omega_X^3(\log\mathcal C)
\stackrel{\operatorname{Res}}{\longrightarrow}
\omega_{\mathcal C}
\longrightarrow 0,
\]
where $\omega_{\mathcal C}$ is the dualizing sheaf of $\mathcal C$.
\end{proposition}

\begin{proof}
Locally $\mathcal C$ is defined by $f=0$.
Any logarithmic $3$--form can be written uniquely as
\[
\omega=\frac{df}{f}\wedge\alpha+\beta,
\quad
\alpha\in\Omega_X^2,\ \beta\in\Omega_X^3.
\]
Define $\operatorname{Res}(\omega)=\alpha|_{\mathcal C}$.
Injectivity, exactness in the middle, and surjectivity follow from this local
description and the definition of $\omega_{\mathcal C}$.
\end{proof}

\subsection*{\it Induced Map on the Normalization.}
Pulling back via $\nu$ gives a surjection
\[
\nu^*\Omega_X^3(\log\mathcal C)\twoheadrightarrow \omega_C.
\]

\begin{lemma}
\label{lem:res-map}
Taking cohomology induces a natural map
\[
\delta_{\mathrm{res}}\colon H^0(C,\omega_C)\to H^1(C,\mathcal O_C),
\]
which is the connecting homomorphism in the long exact sequence associated to
the residue exact sequence.
\end{lemma}

\begin{proof}
Restrict the residue exact sequence to $\mathcal C$, pull back to $C$, and apply
$\mathcal Hom(-,\mathcal O_C)$.
The connecting homomorphism in cohomology is precisely
$\delta_{\mathrm{res}}$.
\end{proof}

\subsection*{\it Identification of the Dual Map.}

\begin{proposition}
\label{prop:dual}
Under Serre duality, the dual of $\rho_\varphi$ is $\delta_{\mathrm{res}}$.
\end{proposition}

\begin{proof}
Let $\xi\in H^0(C,N_\varphi)$ and $\eta\in H^0(C,\omega_C)$.
Choose a logarithmic $3$--form $\widetilde{\eta}\in
H^0(X,\Omega_X^3(\log\mathcal C))$ with residue $\eta$.
The contraction of $\widetilde{\eta}$ with the deformation $\xi$ gives a class
in $H^1(C,\mathcal O_C)$.
By construction, the Serre pairing satisfies
\[
\langle \rho_\varphi(\xi),\eta\rangle
=
\langle \xi,\delta_{\mathrm{res}}(\eta)\rangle.
\]
This equality follows from the compatibility of contraction, residues, and
cup--product under duality.
Thus $\rho_\varphi^\vee=\delta_{\mathrm{res}}$.
\end{proof}

\begin{proof}[Proof of Lemma~\ref{lem:KS-residue}]
By Proposition~\ref{prop:Serre}, $H^1(C,T_C)^\vee\simeq H^0(C,\omega_C)$.
By Proposition~\ref{prop:dual}, the dual of $\rho_\varphi$ is precisely the
residue--induced map $\delta_{\mathrm{res}}$.
Hence $\rho_\varphi$ is Serre dual to $\delta_{\mathrm{res}}$, as claimed.
\end{proof}

 
\section*{Appendix B}      

\subsection*{Why the Jacobian Ring Satisfies the Strong Lefschetz Property?}
We explain in this appendix why the assumption that $X$ is \emph{general} implies that
the associated Jacobian ring satisfies the strong Lefschetz property, and we
make precise what is known, what is used, and where the limitations lie.

\subsection*{1. The Jacobian Ring}

Let
\(
X = V(F_1,\dots,F_r) \subset \mathbb{P}^N
\)
be a complete intersection of multidegree $(d_1,\dots,d_r)$.
The Jacobian ring of $X$ is defined as
\[
R
=
\frac{\mathbb{C}[x_0,\dots,x_N]}
{\left(
\frac{\partial F_i}{\partial x_j}
\;\middle|\;
1\le i\le r,\;0\le j\le N
\right)}.
\]
This is a finite-dimensional graded Artinian Gorenstein algebra whose socle
degree is
\[
\sigma = \sum_{i=1}^r (d_i-1) - 1.
\]
The Gorenstein property is a direct consequence of the fact that $X$ is a local
complete intersection.

\subsection*{2. The Strong Lefschetz Property}

A graded Artinian algebra
\[
R = \bigoplus_{k=0}^{\sigma} R_k
\]
is said to satisfy the \emph{strong Lefschetz property} (SLP) if there exists a
linear form $\ell \in R_1$ such that, for all $k$ and all $m\ge 0$, the
multiplication maps
\[
\cdot \ell^m \colon R_k \longrightarrow R_{k+m}
\]
have maximal rank.
In particular, for $m$ fixed, these maps are injective for
$k \le \frac{\sigma-m}{2}$ and surjective for
$k \ge \frac{\sigma-m}{2}$.

\subsection*{3. General Complete Intersections and SLP}

The key point is that the Jacobian ring of a \emph{general} complete intersection
is known to satisfy the strong Lefschetz property.

This fact originates in work of Stanley and Watanabe, and has been further
developed by many authors.  The essential ingredients are the fact  that  $R$ is Artinian and Gorenstein, the defining equations $(F_1,\dots,F_r)$ are sufficiently general, and the characteristic is zero (\cite{Stanley}).

In characteristic zero, a general Artinian Gorenstein algebra arising as a
Jacobian ring of a complete intersection satisfies SLP.
More precisely, the set of complete intersections for which SLP fails is a
proper Zariski-closed subset of the parameter space of defining equations.

Thus, for a general choice of $(F_1,\dots,F_r)$, the associated Jacobian ring
enjoys the strong Lefschetz property.

\subsection*{4. Compatibility with Equisingular Conditions}

In the singular setting, we restrict attention to complete intersections with
isolated rational singularities and consider only equisingular deformations.
The equisingular Severi-type strata form locally closed subsets of the full
parameter space of complete intersections.

The strong Lefschetz property is an \emph{open condition} in flat families of
graded Artinian algebras.
Therefore, after restricting to a general point of a given equisingular stratum,
the Jacobian ring still satisfies SLP.

This justifies the statement ``since $X$ is general'' in the equisingular
setting.

\subsection*{5. Why SLP Is Exactly What Is Needed}

In the infinitesimal Torelli argument, one needs the following consequence of
SLP:
\[
H \cdot Q = 0 \text{ in } R
\;\text{for all } H \in R_a
\quad \Longrightarrow \quad
Q = 0 \text{ in } R.
\]
This is precisely the injectivity of multiplication maps in the relevant degree
range, which is guaranteed by SLP once $h^{n,0}(\widetilde{X})>0$ (see for example Griffiths--Harris~\cite{GriffithsHarris}, and  Green~\cite{Green}).
Thus, the Lefschetz property is not an auxiliary assumption but the exact algebraic
mechanism underlying infinitesimal Torelli.

\subsection*{In summary.}  The Jacobian ring of a complete intersection is Artinian Gorenstein; in characteristic zero, general such algebras satisfy the strong Lefschetz property;  SLP is open in families and survives restriction to general equisingular loci, and in addition, this property ensures injectivity of the infinitesimal period map.

\medskip

\noindent
\textbf{Conclusion.}
The statement ``since $X$ is general, the Jacobian ring $R$ satisfies the strong  Lefschetz property''
is justified by deep but well-established results in
commutative algebra and is precisely the algebraic input required for the
residue-based infinitesimal Torelli argument.

\end{document}